\documentclass[11pt,fleqn]{amsart} % \documentclass[11pt]{article}

\usepackage[usenames,dvipsnames,condensed]{xcolor}

\usepackage{shuffle}
\usepackage{amsthm}
\usepackage{amsfonts}
\usepackage[english]{babel}
\usepackage[usenames]{xcolor}
\usepackage{graphicx}
\usepackage{soul}
\usepackage{stfloats}
\usepackage{morefloats}
\usepackage{cite}
\usepackage{lscape}
\usepackage{epstopdf}
\usepackage{tikz-cd}
\usepackage{tikz}
\usepackage{bbm}
\usepackage[lite]{amsrefs}

\usepackage{amsmath,amstext,amsopn,amsfonts,eucal,amssymb}

\usepackage{graphicx,wrapfig,url}

\newcommand\Z{{\mathbb Z}}

\newtheorem{theorem}{Theorem}[section]

\newtheorem{lemma}[theorem]{Lemma}
\newtheorem{proposition}[theorem]{Proposition}
\newtheorem{definition}[theorem]{Definition}

\newtheorem{example}[theorem]{Example}

\newtheorem{remark}[theorem]{Remark}

\begin{document}

\title{Heap and Ternary Self-Distributive Cohomology}

\author{Mohamed Elhamdadi} 
\address{Department of Mathematics, 
	University of South Florida, Tampa, FL 33620, U.S.A.} 
\email{emohamed@math.usf.edu} 

\author{Masahico Saito} 
\address{Department of Mathematics, 
	University of South Florida, Tampa, FL 33620, U.S.A.} 
\email{saito@usf.edu} 

\author{Emanuele Zappala} 
\address{Department of Mathematics, 
	University of South Florida, Tampa, FL 33620, U.S.A.} 
\email{zae@mail.usf.edu}

\maketitle

\begin{abstract}
Heaps are para-associative ternary operations bijectively exemplified by %pointed 
groups
via the operation $(x,y,z) \mapsto x y^{-1} z$. 
They are also ternary self-distributive, and have a diagrammatic interpretation in terms of framed links.
Motivated by these properties, we define 
%two kinds of cohomologies, which we call 
para-associative %type zero 
and heap cohomology theories and also a ternary self-distributive cohomology theory with abelian 
heap coefficients. We show that 
one of %type zero 
the heap cohomologies is related to group cohomology via a long exact sequence. Moreover we construct 
%second cohomology 
 %group 
maps between
 second cohomology groups of
normalized group cohomology and heap cohomology, and %we 
show that the latter injects into the ternary self-distributive second cohomology group.
% Relations among cocycles of heap, group, and ternary self-distributive operations are discussed.
We %therefore 
proceed to study heap objects in symmetric monoidal categories providing a characterization of pointed heaps as involutory Hopf monoids in the given category. Finally we prove that heap objects are also ``categorically" self-distributive in an appropriate sense.
%Internalizations are studied in symmetric monoidal category. 
\end{abstract}

\tableofcontents

\section{Introduction}

Self-distributive binary operations, called shelves, and their cohomology theories have been
extensively studied in recent decades with applications to constructing invariants of classical knots and knotted surfaces. 
Ternary self-distributivity and its cohomology has been studied, for example,  in \cite{CEGM,ESZ}, and shown in \cite{ESZ} to have a diagrammatic interpretation in terms of framed links.  
Constructions of ternary self-distributive operations  from binary ones are also given  in \cite{ESZ},
and 
it was shown 
that the (co)homology of ternary operations thus obtained
and the (co)homology of the binary operations used for this construction 
are related through certain (co)chain maps.  

%A heap is a para-associative ternary operation with additional conditions, and is exemplified by a group $G$ with the operation defined by $(x,y,z)\mapsto xy^{-1}z$. 
%%% merged with the following: 
A heap is an abstraction of the ternary operation $a\times b\times c \mapsto ab^{-1}c$ in a group, that allows to ``forget''
which element of the group is the unit. In fact the operation just described extends to a functor that determines an equivalence between the category of 
pointed (i.e. an element is specified) heaps and the category of groups. 
More specific definitions will be given in Section~\ref{sec:prelim}.
Heaps have been studied in algebra and algebraic geometry under the name of torsors. %\cite{Kon}. % unclear %%% \cite{Okubo1, Okubo2}. 
They appeared in knot theory in relation to region colorings as well (see, for example, \cite{NOO, Maciej}). 
It was pointed out in \cite{CEGM} that heaps are ternary self-distributive. 
A diagrammatic representation of such  operations was given in terms of framed links in \cite{ESZ}. 
In the present paper, we introduce and develop a cohomology theories for heaps, a homology theory for ternary self-distributivity (TSD) %homology 
with heap coefficients,  and present relations between them.
%We then present their relations.
 
Specifically, the para-associativity comes in three types which we call type 0, 1 and 2, where
 type 0 takes the familiar form $[ [x,y,z]u,v]=[x,y,[z,u,v]]$. In this case we define a chain complex in a similar manner to 
 group homology, and establish a long exact sequence relating the two %co
 homology theories. %% (Proposition~\ref{4.1}).
 For the other types, however, such an analogue in general dimensions is elusive.
 Thus we take an approach of defining low dimensional cochain maps from point of view of extensions, and 
 %for 
 with the goal 
 of applying them to  the TSD cohomology. In particular, the second cohomology group classifies isomorphism classes of 
 heap extensions. %% (Proposition~\ref{3.16}). %%
We also provide relations between 2-cocycles for groups and para-associative 2-cocycles  of types 1 and 2. 
Our motivation is to use this relation to construct TSD cocycles from group cocycles.

The main results of the paper are introducing and studying a TSD homology theory with abelian group heap coefficients. 
This differs from   \cite{CEGM,ESZ}
in that the heap structure of the coefficient is essentially used in the definition of the chain complex. 
The definition again provides the classification of extensions with heap coefficients by the second TSD cohomology 
group. %%  (Prposition~\ref{5.7}). %%
We then present an injective  map from heap to TSD cohomology groups $H^2_{\rm H}(X, A) \rightarrow H^2_{\rm SD}(X,A)$ in dimension 2. %% (Theorem~\ref{6.2}). %%
Non-trivial examples are provided throughout.

Binary self-distributive operations have been studied in relation to the Yang-Baxter operators 
though tensor categories (e.g., \cite{CCES}).
 In \cite{ESZ} a diagrammatic interpretation of TSD was given in terms of framed links,
providing set-theoretic Yang-Baxter operations. It is, then, a natural question whether the constructions of TSD operations
from heaps generalize to monoidal categories.
For this goal we introduce category versions of heaps and TSD operations,
and prove that a heap object in symmetric monoidal category is also a TSD object. %% (Theorem~\ref{7.11}).

The paper is organized as follows. 
After a review of  basic materials of heaps in Section~\ref{sec:prelim}, a cohomology theory 
and extensions by 2-cocycles are presented in Section~\ref{sec:heapcoh}. 
Using the bijection between pointed heaps and groups, 
constructions of heap 2-cocycles from group 2-cocycles, and vice versa, are discussed in Section~\ref{sec:gpheap}. 
A cohomology theory of ternary self-distributive (TSD) operations with abelian group heap coefficients is introduced 
in Section~\ref{sec:TSD}, and the extension theory is built from 2-cocycles.
A construction of TSD 2-cocycles from heap 2-cocycles is given in Section~\ref{sec:heapTSD}, 
and internalization in the symmetric monoidal category is discussed in Section~\ref{sec:cat}.

\section{Basic Review of Heaps}\label{heaprev} \label{sec:prelim}

In this section we recall the definitions of para-associative structure on a set and introduce the nomenclature that will be followed throughout the rest of the article. 
Given a set with a ternary operation $[ - ]$, we call the equalities 
%$$ % 
\begin{eqnarray*}
%\begin{array}{rcl}
{} [  [ x_1, x_2, x_3 ], x_4, x_5  ] & = & [  x_1,x_2, [  x_3, x_4, x_5 ]  ]  \\
{}  [  [ x_1, x_2, x_3 ], x_4, x_5  ] & = & [  x_1,[  x_4, x_3, x_2 ]  ,  x_5 ]  \\
{}  [  x_1,x_2, [  x_3, x_4, x_5 ]  ] &=& [  x_1,[  x_4, x_3, x_2 ],  x_5  ]  
%\end{array}
%$$ % 
\end{eqnarray*}
the type 0, 1, 2 para-associativity (or simply %equality
para-associativity), 
respectively.
Observe that any pair of equalities of type 0, 1 or 2 implies that the remaining one also holds.  
If a ternary operation satisfies all types of para-associativity, then it is called para-associative.
We call the condition $[x,x,y]=y$ and $[x, y, y ] = x$ the degeneracy (conditions).

\begin{definition}
	{\rm 
		A heap is a non-empty set with a ternary operation satisfying 
		para-associativity % type 0, 1, 2 equalities 
		and degeneracy conditions.
	}
\end{definition}

We mention that a structure satisfying the para-associativity conditions only, is called {\it semi-heap} in Chapter 8 of \cite{Vag}.

A typical example of a heap is a group $G$ where the ternary operation is given by $[x,y,z]=xy^{-1}z$,
which we call a {\it group heap}. If $G$ is abelian, we call it an {\it abelian (group) heap}. 
Conversely, 
given a heap $X$ with a fixed element $e$, %then 
one defines a binary operation on $X$ by $x*y=[x,e,y]$ which makes $(X,*)$ into a group with $e$ as the identity, and the inverse of $x$ is $[e,x,e]$ for any $x \in X$.

We refer the reader to the classical reference \cite{BH}, chapter IV, where it can also be found a short 
historical background and a 
description in terms of universal algebra.  
%Heaps are also known under different names
%such as {\it torsor} and {\it groud}. 
In \cite{Sk} a quantum version of heap was introduced and it has been shown, in analogy to the ``classical''  case, that the category of quantum heaps is equivalent to the category of pointed Hopf algebras. Further developments of the thematics introduced in \cite{Sk} can also be found in \cite{Gr,Sch}. Other  
sources include  
\cite{Kon,BS}. We observe that the definition of quantum heap given in \cite{Sk} is in some sense dual to the notion of heap object in a symmetric monoidal category, that we introduce in Section \ref{sec:cat}. Our heap objects in symmetric monoidal categories are much in the same spirit as in the definition of non-commutative torsor treated in \cite{BS}.

\section{Heap Cohomology}\label{sec:heapcoh}

In this section we introduce %the
a heap cohomology.
Let $X$ be a set with a ternary operation  $[  -  ]$ and $A$ be an abelian group.
The $n$-dimensional {\em cochain group} $C^n_{\rm PA}(X, A)$, 
  is the 
  group of functions $\{ f: X^{2n-1} \rightarrow A\}$
  for $n = 1,2 $.

\begin{definition}\label{def:PA1cocy}
	{\rm 
	Let $X$ be a set with a para-associative ternary operation $[  -  ]$ and $A$ be an abelian group. 
	Then the 1-dimensional coboundary map $\delta^1: C^1_{\rm PA}(X, A) \rightarrow C^2_{\rm PA} (X, A)$ is defined for $f \in C^1_{\rm PA} (X, A)$  by 
		%The heap $1$-cocycle condition is defined to be:
		$$ 
		\delta^1f(x,y,z) = f([x,y,z]) - f(x) +f(y) -f(z).
		$$
The kernel $Z^1_{\rm PA} (X, A)$ of $\delta^1$ is called the 1-dimensional cocycle group.
In this case we define 1-dimensional cohomology group $H^1_{\rm PA} (X, A)$ to be $Z^1_{\rm PA} (X, A)$.
}
\end{definition}

We observe that $f \in Z^1_{\rm PA} (X, A)$ if and only if $f$ is a para-associative homomorphism from $X$ to 
$A$ regarded as an abelian heap. 
We determine $Z^1_{\rm PA} (X, A)$ for the following two examples.
\begin{example}\label{ex:Z1heap}
	{\rm
Consider $\Z_2$ 
 with the abelian heap operation 
 $[x,y,z] = x+y+z$. 
 We  
 compute the group $Z^1_{\rm PA} (\Z_2, \Z_2)$.
  Given three variables $x,y$ and $z$, at least two of them need to coincide. Consider the case when  $x = y$, the $1$-cocycle condition becomes $f([x,x,z]) = f(z)$ which is 
   satisfied. 
   The other cases are analogous. It follows that $ Z^1_{\rm PA} (\Z_2, \Z_2) = C^1_{\rm PA}(\Z_2, \Z_2) \cong \Z_2\oplus \Z_2$.  
   This operation is also ternary self-distributive, see Lemma \ref{lem:heapTSD}.
}
\end{example} 

\begin{example}
	{\rm
We proceed  to compute $Z^1_{\rm PA} (\Z_3, \Z_n)$, where $\Z_3$ is given the same ternary operation as before: $[x,y,z] = x+y+z$. Observe that this operation does not define a heap, but it is para-associative. Take $x=y=1$ in the $1$-cocycle condition. We obtain $f(z+2) = f(z)$, which implies that $f$ is  
the constant map. It follows that $Z^1_{\rm PA} (\Z_3, \Z_n) \cong \Z_n$ for all odd integers $n$.
}
\end{example}

\begin{definition}\label{def:PA2}
{\rm
Define $C^3_{\rm PA(i)} (X, A)$ for $i=0,1,2$ to be three isomorphic copies of the %free abelian group generated by
abelian group of functions $\{ f: X^{5} \rightarrow A\}$.
The  
2-dimensional coboundary map $\delta^2_{(i)} : C^2_{\rm PA} (X, A) \rightarrow C^3_{\rm PA(i)} (X, A)$ 
of type $i=0,1,2$, respectively, are defined by 
\begin{eqnarray*}
%\lefteqn{ 
\delta^2_{(0)} \eta (x_1, x_2, x_3, x_4, x_5)   % } \\
&=& \eta (x_1, x_2, x_3) + \eta ( [ x_1, x_2, x_3],  x_4, x_5 ) \\
& & -  \eta  (x_3, x_4, x_5)  -   \eta  (x_1, x_2,  [x_3, x_4, x_5] )) , \\
%\lefteqn{  
\delta^2_{(1)} \eta  (x_1, x_2, x_3, x_4, x_5)  % } \\
&=&  \eta  (x_1, x_2, x_3) + \eta ( [ x_1, x_2, x_3],  x_4, x_5 ) \\
& & +  \eta  (x_4, x_3, x_2) -   \eta (x_1, [x_4, x_3, x_2], x_5) ) , \\
% \lefteqn{ 
\delta^2_{(2)} \eta (x_1, x_2, x_3, x_4, x_5) %  }\\
&=&    \eta  (x_3, x_4, x_5 )  +   \eta  (x_1,  x_2, [ x_3, x_4 , x_5] ) \\
& &  +  \eta  (x_4, x_3, x_2) -   \eta  (x_1, [x_4, x_3, x_2], x_5) ).   \end{eqnarray*}

}
\end{definition}

Direct calculations give the following.

\begin{lemma} \label{lem:d2d1}
If $X$ has a para-associative operation $[ -  ]$, then 
$\delta^2_{(i)} \delta^1 =0$ for $i=0,1,2$. 
\end{lemma}

\begin{definition}
{\rm
Let  $X$ be a set with a para-associative operation $[ -  ]$ and let $A$ be an abelian group.
Define $C^3_{\rm PA}(X, A)= C^3_{\rm PA(1)}(X,A)  \oplus C^3_{\rm PA(2)}(X,A)$.
Then $\delta^2=\delta^2_{(1)} \oplus \delta^2_{(2)}$ defines a homomorphism
$C^2_{\rm PA}(X,A) \rightarrow C^3_{\rm PA}(X, A)$.
Define the group of $2$-cocycles
% $2^{\rm nd}$ cocycle group 
 $Z^2_{\rm PA}(X,A) $ by 
${\rm ker} ( \delta^2 ) $. 
Define the $2^{\rm nd}$ coboundary  group $B^2_{\rm PA}(X,A) $ by ${\rm im}(\delta^1)$. 
Then the  $2^{\rm nd}$ cohomology  group is defined as usual: $H^2_{\rm PA}(X,A) =
Z^2_{\rm PA}(X,A) /  B^2_{\rm PA}(X,A)$.
}
\end{definition}

\begin{definition}\label{1cocy}
{\rm
Let  $X$ be a set with a para-associative operation $[ -  ]$ and let $A$ be an abelian group.
A 2-cocycle $\eta \in Z^2_{\rm PA}(X,A) $ is said  to satisfy the {\it degeneracy condition}
if the following holds for all $x, y \in X$:
$\eta (x, x, y)=0 = \eta ( x,y,y)$.
}
\end{definition}

We observe that 2-coboundaries $\delta^1 f$ satisfy the degeneracy condition. 

\begin{definition}
{\rm
Let  $X$ be a heap  and $A$ be an abelian group.
The $2^{\rm nd}$ heap cocycle group $Z^2_{\rm H}(X,A) $ is defined as 
the subgroup of $Z^2_{\rm PA}(X,A) $ %generated by the
consisting of 
2-cocycles that  satisfy the degeneracy conditions.
 The  $2^{\rm nd}$ heap cohomology  group $H^2_{\rm H}(X,A) $ is defined 
as the quotient  $  Z^2_{\rm H}(X,A) /  B^2_{\rm PA}(X,A)$. 
}
\end{definition}

\begin{example}\label{ex:H2heap}
{\rm
Let $X = \Z_2$ with group heap operation and $A = \Z_2$. 
Computations show that $\eta \in Z^2_{\rm PA}(X,A) $ if and only if $\eta$ satisfies the following set of equations:
\begin{eqnarray*}
 & & 
\eta (0,0,0)= \eta (0,0,1)  = \eta (1,0,0) ,\\
& & 
\eta (1,1,1) = \eta (1,1,0) = \eta (0,1,1) , \\
& & 
\eta (0,0,0) + \eta(1,1,1) %% changed for symmetry :\eta (0,0,1) + \eta(0,1,1)  
+ \eta(0,1,0) + \eta(1,0,1) =0 . 
\end{eqnarray*}
Express $\eta=\sum \eta(x,y,z) \chi_{(x,y,z)} $ by characteristic functions $ \chi_{(x,y,z)}$.
By setting $\eta (0,0,0)=a$, $ \eta (1,1,1)=b$ and  $\eta(0,1,0) =c$, the last equation above 
implies $\eta (1,0,1) = - (a+b+c)$.  
Then $\eta$ is expressed as 
\begin{eqnarray*}
\eta&=& 
a( \chi_{ (0,0,0)} + \chi_{ (0,0,1)  } + \chi_{ (1,0,0) } - \chi_{( 1,0,1) }  ) \\
& + & 
b (  \chi_{ (1,1,1) } + \chi_{ (1,1,0)  } + \chi_{ (0,1,1) } - \chi_{( 1,0,1) } )  \\
& + & 
c (  \chi_{ (0,1,0)} - \chi_{ (1,0,1)  } ) . 
\end{eqnarray*}
Since the group of coboundaries is zero  from Example \ref{ex:Z1heap}, it follows that 
 $H^2_{\rm PA}(\Z_2,\Z_2) = Z^2_{\rm PA}( \Z_2,\Z_2 )\cong \Z_2 \oplus \Z_2 \oplus \Z_2 $. 
Since the degeneracy condition implies $a=b=0$, we have  $H^2_{\rm H}(\Z_2,\Z_2) \cong \Z_2$.

}
\end{example}

\begin{definition}\label{def:heapext}
	{\rm
		Let  $X$ be a heap, $A$ be an abelian group and $\eta : X\times X\times X \rightarrow A$ be a $2$-cochain. We define the heap extension of $X$ by the $2$-cochain $\eta$ with coefficients in $A$, denoted $X\times_{\eta} A$, as the cartesian product $X\times A$ with ternary operation given by:
		$$ 
		[(x,a),(y,b),(z,c)] = ([x,y,z], a- b + c + \eta(x,y,z)).
		$$
	}
\end{definition}

%% added
Although direct computation gives the following lemma, it is one of the motivations  
of the definition of the heap differential maps. 

\begin{lemma}\label{ext}
The abelian extension by a $2$-cochain $\eta$ satisfies 
the equality of type 1, 2, and degeneracy if and only if $\eta$ is a heap $2$-cocycles of type 1, 2, and with degeneracy condition, respectively. 
In particular, a $2$-cochain  $\eta$ defines 
%an extension heap 
a heap extension
if and only if it satisfies all conditions.
\end{lemma}

%\begin{proof}
%We prove the lemma for type 1 equality and type 1 cocycle condition, the remaining cases being completely analogous.\\
%Let $\eta : X\times X\times X \longrightarrow A$ be a $2$-chain. We have:
%\begin{eqnarray*}
%\lefteqn{[[(x,a),(y,b),(z,c)], (u,d),(v,e)]} \\
%&=&
%  ([[x,y,z],u,v], a-b +c +\eta(x,y,z) - d +e + \eta([x,y,z],u,v)).
%\end{eqnarray*}
%It also holds:
% \begin{eqnarray*}
% 	\lefteqn{[(x,a),[(u,d),(z,c),(y,b)],(v,e)]}\\
% 	&=& ([x,[u,z,y],v], a-d+c-b-\eta(u,z,y)+e+\eta(x,[u,z,y],v).
% 	\end{eqnarray*}
% The two terms coincide (and hence equality 1 holds) if and only if $\eta$ satisfies the $2$-coclycle condition of type 1. 
%\end{proof}

\begin{example}
{\rm
The following  is a common construction applied to the heap.
Let $0 \rightarrow A \stackrel{\iota}{\rightarrow}  E \stackrel{\pi }{\rightarrow} G \rightarrow 0 $ be 
a short exact sequence of abelian groups, and $s: G \rightarrow E$ be a set-theoretic section ($\pi s = {\rm id}$). 
Since $s$ is a section, we have that $  s(x) - s(y) + s(z) - s( [x,y,z] ) $ is in the kernel of $\pi$ for all $x,y,z \in G$, so that there is 
$\eta: G \times G \times G \rightarrow A$ such that 
$$\iota \eta (x, y, z) =  s(x) - s(y) +s(z) - s( [x,y,z] )   . $$
Then computations of the two 2-cocycle conditions and the degeneracy conditions give the following.
}
\end{example}

\begin{lemma}
	$\eta \in Z^2_{\rm H} (G, A)$.
\end{lemma}

\begin{example}\label{ex:heapextmodn} 
	{\rm 
		For a positive integer $n>0$, let $0 \rightarrow \Z_n \stackrel{\iota}{\rightarrow}  \Z_{n^2}  \stackrel{\pi }{\rightarrow} \Z_n \rightarrow 0 $
		be as above, where $s(x) \ {\rm mod}(n^2)  = x$, representing elements of $\Z_{m}$ by $\{ 0, \ldots, m-1 \}$. 
		Then for all $x, y, z \in G=\Z_n$, 
		$\iota \eta(x, y, z)$ is divisible by $n$ in $E=\Z_{n^2}$, so that the value of $\eta$ is computed by
		$\eta  (x,y,z) = \iota \eta(x, y, z) / n $. For example, 
		for $n=3$, $\eta(2, 0, 2) = [ s(2) - s(0) + s(2) - s( [2,0,2] )   ] / 3 = 1 \in \Z_3$. We will show in Example~\ref{ex:nontheap} that $[\eta]\neq 0$ and therefore $H^2_{H}(\Z_3,\Z_3)$ is nontrivial. 
	}
\end{example}

\begin{definition}\label{heapextmor}
	{\rm
Let $X\times_{\eta}	A$ and $X\times_{\eta'}	A$ be two heap extensions with coefficients in an abelian group $A$, by two $2$-cocycles $\eta$ and $\eta'$ of type 1,2 and with degeneracy condition. We define a morphism of extensions, indicated by $\phi:  X\times_{\eta}	A \longrightarrow X\times_{\eta'}	A$, to be a morphism of heaps making the following diagram (of sets) commute.
$$
\begin{tikzcd}
	X\times A\arrow[rr,"\phi"]\arrow[dr]& & X\times A\arrow[dl]\\
	& X &
\end{tikzcd}
$$
An invertible morphism of extensions is also called isomorphism of extensions, which induces an equivalence relation,
and its equivalence classes are called isomorphism classes.
}
\end{definition}

%% added
The following is one of the motivations and significance of our definition of the heap differential in dimension 2.

\begin{proposition}\label{pro:heapext}
	There is a bijective correspondence between isomorphism classes of heap extensions by $A$, and the second heap cohomology group $H^2_{\rm H}(X;A)$.
\end{proposition}
\begin{proof}
	Standard arguments, similar to the group-theoretic case, give the result.
%	Let us first assume that $\phi: X\times_{\eta} A\longrightarrow X\times_{\eta'}	A$ is an isomorphism of extensions. Since both $(x,a)$ and $\phi(x,a)$ project onto the same element $x$, we can write the map $\phi$, as $\phi(x,a) = (x, a + f(x))$, for some set-theoretic function $f: X \longrightarrow A$. We therefore have
%	$$
%	\phi([(x,a),(y,b),(z,c)]) = ([x,y,z], a-b+c+\eta(x,y,z) + f([x,y,z])).
%	$$
%	Similarly, we have
%	$$
%	[\phi(x,a), \phi(y,b),\phi(z,c)] = ([x,y,z], a + f(x) -b - f(y) + c + f(z) + \eta'(x,y,z)).
%	$$
%	Since $\phi$ is a morphism of heaps,  the two expressions are equal. It follows that $\eta = \eta' + \delta^1f$ and therefore $\eta$ and $\eta'$ are in the same cohomology class.
%	Conversely,  if $\eta$ and $\eta'$ represent the same cohomology class, they differ by $\delta^1 f$, for some $1$-cochain $f$. Define $\phi$ to be $\phi(x,a) = (x, a + f(x))$. Using the same equations as above we see that $\phi$ is a morphism of heaps. The fact that it is an isomorphism of extensions comes from the fact that it is easily seen to be bijective and obviously $\pi_1(x,a) = x =\pi_1 \phi (x,a)$. 
\end{proof}

\begin{lemma}\label{lem:type0eqn}
Let $\eta_0$ be  a heap $2$-cocycle of type 0 that satisfies the degeneracy condition. Then  the following equality holds
$$\eta_0 (x_1, x_2, x_3) +\eta_0 ( [x_1, x_2, x_3], x_3, x_2)=0. $$
\end{lemma}

\begin{proof} The $2$-cocycle condition applied to $\eta_0(x_1,x_2,x_3,x_3,x_2)$ becomes
\begin{eqnarray*}
\lefteqn{\delta^2_{(0)}\eta_0(x_1,x_2,x_3,x_3,x_2)}\\
&=& \eta_0(x_1,x_2,x_3) + \eta_0([x_1,x_2,x_3],x_3,x_2)\\
&&- \eta_0(x_3,x_3,x_2) - \eta_0 (x_1,x_2,[x_3,x_3,x_2]).
\end{eqnarray*}
By applying the degeneracy condition and the degeneracy heap axiom, we obtain the result.
\end{proof}

%%% The following may be unused.%%%%%%% Or check that this is type 0 cocycle of type 0 homology below. 
%\begin{definition}
%{\rm
%Let $(\zeta_1, \zeta_2) $ be a pair of heap $3$-cochains.
%Define a heap $3$-cochain of type $0$ corresponding to $(\zeta_1, \zeta_2)$ by 
%$$\zeta_0 (x_1, x_2, x_3, x_4, x_5)=
%\zeta_1 (x_1, x_2, x_3, x_4, x_5) - \zeta_2 (x_1, x_4, x_3, x_2, x_5 ) . $$
%}
%\end{definition}

%% added
Towards extending cohomology theory to general dimensions, we propose the following 3-differentials.

\begin{definition}
{\rm
Let $X$ be a set with para-associative operation $[ - ]$ and let $A$ be an abelian group.
Let $C^4_{\rm PA(i)} (X,A)$ be three isomorphic copies of the abelian group of  functions $\{ f: X^{7} \rightarrow A\}$  for $i=1,2,3$.
Let $C^4_{\rm PA} (X,A)=\oplus_{i=1,2,3} C^4_{\rm PA(i)} (X,A) $.
For $(\zeta_1, \zeta_2) \in C^3_{\rm PA} (X,A)=C^3_{\rm PA(1)} (X,A) \oplus C^3_{\rm PA(2)} (X,A)$, 
define $\delta^3_{(i)}: C^3_{\rm PA} (X,A) \rightarrow C^4_{\rm PA(i)} (X,A) $, for $i=1,2,3$, as follows.
\begin{eqnarray*}
 \lefteqn{ \delta^3_{(1)} (\zeta_1, \zeta_2)   (  x_1, x_2, x_3, x_4, x_5, x_6 , x_7 ) }\\
&=&    \zeta_1( [x_1, x_2, x_3], x_4, x_5, x_6, x_7) 
+ \zeta_1 ( x_1, x_2, x_3, [x_6, x_5, x_4], x_7 ) \\
& & -  \zeta_1 (  x_6, x_5, x_4, x_3, x_2) %%%%% The sign of this term was changed.
-  \zeta_1 ( x_1, x_2, x_3, x_4, x_5) \\
& & + \zeta_2 ( x_1, x_2, x_3, x_4, x_5) 
- \zeta_1 (x_1, x_2, [x_3, x_4, x_5] , x_6 , x_7)
 ,\\
\lefteqn{  \delta^3_{(2)} (\zeta_1, \zeta_2)  (  x_1, x_2, x_3, x_4, x_5, x_6 , x_7 )}\\
&= &
 \zeta_2 (x_1,x_2, x_3, x_4,  [x_5, x_6, x_7] ) 
+ \zeta_2 (x_1,  [x_4, x_3, x_2 ], x_5, x_6, x_7)  \\
& &
- \zeta_2 ( x_6, x_5, x_4, x_3, x_2)  %%%%% The sign of this term was changed.
-  \zeta_2 (x_3, x_4, x_5,  x_6, x_7 ) \\
 & &   
 + \zeta_1 ( x_3, x_4, x_5, x_6, x_7 )
  - \zeta_2 ( x_1, x_2, [x_3,  x_4, x_5],  x_6, x_7)  , 
  \\
  \lefteqn{  \delta^3_{(3)} (\zeta_1, \zeta_2) (  x_1, x_2, x_3, x_4, x_5, x_6 , x_7 ) } \\
&= &  
 \zeta_2 ( [x_1, x_2, x_3], x_4, x_5, x_6, x_7) 
+  \zeta_1 ( x_1, x_2, x_3, [x_6, x_5, x_4], x_7) \\
& & 
 +  \zeta_1 ( x_6, x_5, x_4,  x_3, x_2 ) 
 -\zeta_1 ( x_1, x_2, x_3, x_4, [ x_5, x_6, x_7] )  \\
& &  - \zeta_2 (x_1,  [ x_4, x_3, x_2] , x_5, x_6 ,    x_7 )   
- \zeta_2 ( x_6,  x_5, x_4, x_3, x_2)  .
\end{eqnarray*}
Then define $\delta^3:=\oplus_{i=1,2,3} \delta^3_{(i)}: C^3_{\rm PA} (X,A) \rightarrow C^4_{\rm PA} (X,A)$.

}
\end{definition}

\begin{figure}[htb]
\begin{center}
\includegraphics[width=2.5in]{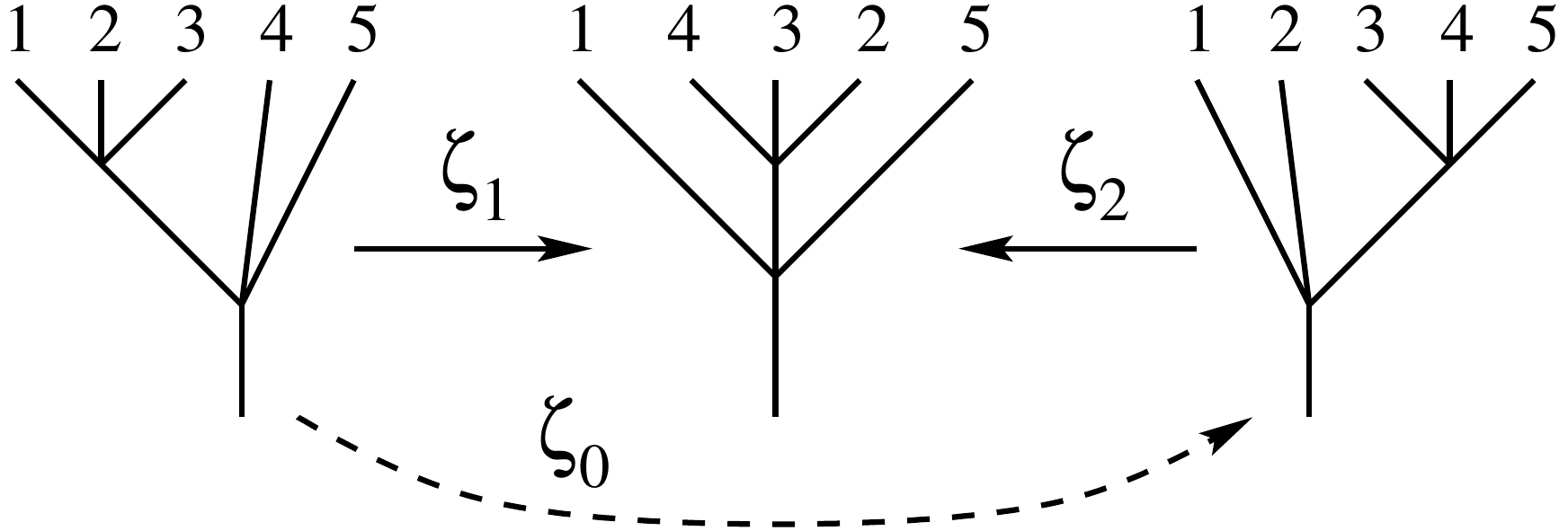}
\end{center}
\caption{Heap 3-cocycle notations}
\label{2cocy}
\end{figure}

Let $X$ be a heap, and let $x_i \in X$ for $i=1, \ldots, 5$. 
We utilize the following diagrammatic representations of heap 3-cocycles in Theorem~\ref{thm:3cocy}.
In Figure~\ref{2cocy}, 3-cocycles are associated to changes of diagrams. 
The three tree diagrams with top vertices labeled represent the elements in the equality 
$$[ [ x_1, x_2, x_3], x_4, x_5 ] = [  x_1,[ x_4, x_3, x_2] , x_5 ] = [  x_1, x_2, [x_3, x_4, x_5 ] ], $$
from left to right, respectively.
The 3-cocycle $\zeta_1(   x_1, x_2, x_3, x_4, x_5 )$
(resp. $\zeta_2(   x_1, x_2, x_3, x_4, x_5 )$)  is associated to  the change from left to middle
(resp. right to middle)  tree diagrams as depicted by the solid arrows. The 3-cocycle
$\zeta_0 (  x_1, x_2, x_3, x_4, x_5 )$ is associated to the change from left to right, and depicted by the dotted arrow.

\begin{figure}[htb]
\begin{center}
\includegraphics[width=5in]{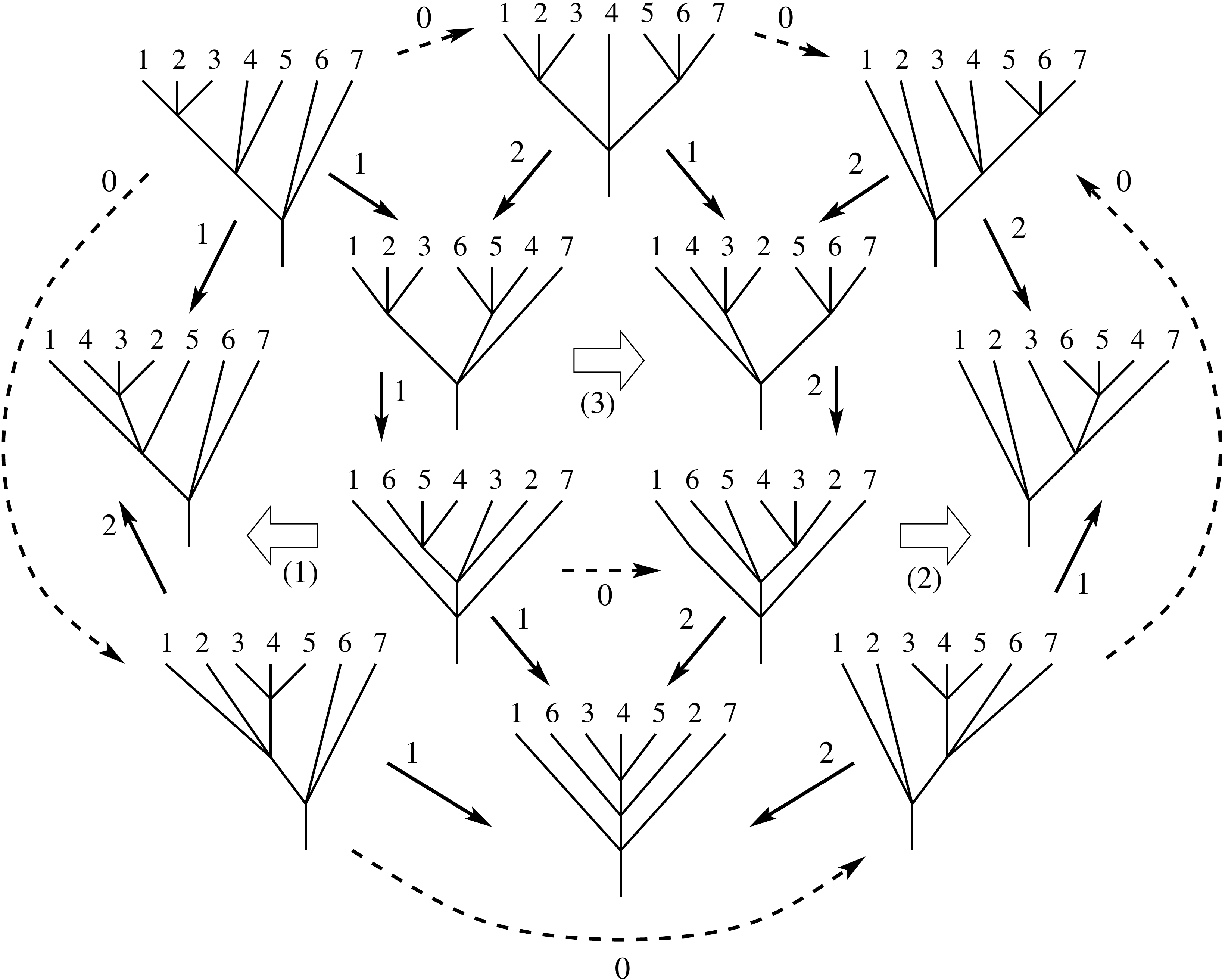}
\end{center}
\caption{Heap 3-cocycle conditions}
\label{3cocy}
\end{figure}

In Figure~\ref{3cocy}, the 3-cocycle conditions are represented by diagrams with 7 elements. 
In the figure, labeled arrows represent 3-cocycles as described above. In the middle, there is a hexagon 
formed by labeled arrows, and has double arrow labeled by $(3)$.
%$(0)$. 
This hexagon represents the differential $\delta^3_{(3)}$. 
The definition of the differentials, as well as the  proof of Theorem~\ref{thm:3cocy}, are aided by this figure.

\begin{theorem} \label{thm:3cocy}
The composition $\delta^3 \delta^2$ vanishes. 
\end{theorem}

\begin{proof}
This follows by proving, for $\eta \in C^2_{\rm PA}(X, A)$ and $\zeta_i=\delta^2_{(i)} \eta $ for $i=1,2$,
that $\delta^3_{(j)} (\zeta_1, \zeta_2)=0$ for $j=1,2,3$. 
For $ \delta^3_{(3)} (\zeta_1, \zeta_2) =0$, first we compute positive terms:
\begin{eqnarray*}
\lefteqn{   \zeta_2 ( [x_1, x_2, x_3], x_4, x_5, x_6, x_7) } \\
 \lefteqn{   +  \zeta_1 ( x_1, x_2, x_3, [x_6, x_5, x_4], x_7) 
 +  \zeta_1 ( x_6, x_5, x_4,  x_3, x_2 ) } \\
 &=&
\{  \eta (  x_5, x_6, x_7 ) +\eta( [x_1, x_2, x_3], x_4,[ x_5, x_6, x_7])   \\
 & &
 - \underline{ \eta ( x_6, x_5, x_4 ) } -  \underline{  \eta( [x_1, x_2, x_3], [ x_6, x_5, x_4], x_7 ) }  \} \\
 & & 
  + \{ \eta (  x_1, x_2, x_3 ) +  \underline{ \eta( [ x_1, x_2, x_3],  [ x_6, x_5, x_4],  x_7 )}  \\ 
& &
 - \underline{ \eta (  [x_6, x_5, x_4], x_3, x_2 ) } -  \eta ( x_1, [  [x_6, x_5, x_4],    x_3,x_2] , x_7 )   \} \\ 
 & & 
  + \{ \underline{  \eta (  x_6 ,  x_5, x_4 ) } + \underline{\eta (  [x_6, x_5, x_4], x_3, x_2 )} \\
& & -  \eta( x_3, x_4, x_5 )  - \eta( x_6, [ x_3, x_4, x_5 ], x_2)  \} 
\end{eqnarray*}
where canceling terms are underlined.
For the remaining terms, one computes 
\begin{eqnarray*}
\lefteqn{  \zeta_1 ( x_1, x_2, x_3, x_4, [ x_5, x_6, x_7] )  }\\
\lefteqn{   + \zeta_2 (x_1,  [ x_4, x_3, x_2] , x_5, x_6 ,    x_7 )   
+ \zeta_2 ( x_6,  x_5, x_4, x_3, x_2) } \\
&=& \{ \eta ( x_1, x_2, x_3 ) +   \eta (  [  x_1, x_2, x_3 ], x_4, [ x_5, x_6, x_7] )    \\
& & 
-  \underline{ \eta ( x_4, x_3, x_2 ) } -  \underline{ \eta ( x_1, [   x_4, x_3, x_2  ] ,  [ x_5, x_6, x_7]  ) }  \} \\
& & 
+ \{ \eta ( x_5, x_6, x_7 ) +   \underline{ \eta ( x_1, [x_4, x_3, x_2],   [ x_5, x_6, x_7 ]  ) }  \\
& & 
-   \underline{ \eta ( x_6, x_5, [  x_4, x_3, x_2]  )}  -   \eta (x_1, [ x_6, x_5, [  x_4, x_3, x_2] ] , x_7 )   \} \\
& & 
+ \{  \underline{ \eta ( x_4, x_3, x_2 ) } +   \underline{ \eta ( x_6, x_5, [  x_4, x_3, x_2]  ) } \\
& & 
- \eta( x_3, x_4, x_5 )   - \eta ( x_6, [x_3, x_4, x_5 ], x_2 ) \} 
 \end{eqnarray*}
and all  terms cancel.

The conditions  $ \delta^3_{(1)} (\zeta_1, \zeta_2) =0$ and
 $ \delta^3_{(2)} (\zeta_1, \zeta_2) =0$ follow similarly from direct computations.
\end{proof}

For a type $0$ condition, a chain complex is defined in a similar manner to the group homology as follows.

\begin{definition}\label{0cocy}
{\rm 
Let $X$ be a set with a type $0$ para-associative ternary operation $[ - ]$.
The $n$-th (type $0$) para-associative (PA)  chain group, 
denoted by $C^{\rm PA}_{n}(X) $, is defined to be the free abelian group on tuples
$(x_1, \ldots, x_{2n-1})$, $x_i \in X$, and the boundary map 
$\partial_{n}^{(0)}: C^{\rm PA}_{n}(X) \rightarrow C^{\rm PA}_{n-1}(X)$ is defined by 
\begin{eqnarray*}
\lefteqn{ \partial_n^{(0)} (x_1, \ldots, x_{2n-1} ) } \\
 &=& -(x_3,\ldots , x_{2n-1}) \\
 && + \sum_{i=1}^n (-1)^{i+1} (x_1, \ldots, x_{2i-2}, [x_{2i-1}, x_{2i}, x_{2i+1} ], x_{2i+2}, \ldots, x_{2n-1} )\\
 && + %sign
   (-1)^{n+1} (x_1, \ldots , x_{2n-3}) %%% sign  (-1)^{n+1} added,  Check
 \end{eqnarray*}
for $n\geq 2$ and $\partial_2^{(0)} (x_1, x_2, x_3)=([x_1, x_2, x_3])-(x_1)+(x_2)-(x_3)$. % changed partial_1 to 2 
}
\end{definition}

It is straightforward to verify that the boundary maps defined above do indeed satisfy the differential condition and define therefore a chain complex.
The dual cochain groups  with coefficient group $A$  and their dual differential maps 
coincide with those in Definitions~\ref{def:PA1cocy} and \ref{def:PA2} for (type $0$) cochain maps. 

\begin{definition}\label{def:type0}
	{\rm
	The homology of the chain complex introduced in Definition \ref{0cocy} is called {\it type 0} para-associative (PA) homology, and written $H^{(0)}_n(X)$. 
}
\end{definition}

%%%% The following is too standard to state as a lemma, so just mention the fact.

%We 
%state the following standard argument to clarify its application in what follows.
%\begin{lemma}\label{lem:cycocy}
%	Suppose $\phi$ is a PA (or heap) $2$-cocycle representing the zero class in the PA (or heap) second cohomology group. Then, if $\alpha$ is a heap $2$-cycle as in Definition \ref{def:type0} we have $\phi(\alpha) = 0$.
%\end{lemma}	
%\begin{proof}
%	Since $[\phi] = 0$ there exists $f\in C^1(X)$ such that $\delta^1f = \phi $ and, since $\delta^1$ is dual to the heap homology chain map $\partial_2^{(0)}$, we have the %chain of 
%	equalities
%	$$
%	\phi(\alpha) = \delta^1f(\alpha) = f(\partial_2^{(0)}  \alpha) = f(0) = 0
%	$$
%as desired. \end{proof}

%%% add:

We note that $\partial_2^{(0)}$ defined above  is dual to $\delta^1$ in 
Definition~\ref{def:PA1cocy}. Therefore if $\phi$ is a 2-coboundary and 
$\alpha$ is a 2-cycle, then 
$$
	\phi(\alpha) = \delta^1f(\alpha) = f(\partial_2^{(0)}  \alpha) = f(0) = 0.
$$
Hence the standard argument applies that if $\phi(\alpha) \neq 0$ for  a 2-cycle $\alpha$
then $\phi$ is not nullcohomologous.

\begin{example}\label{ex:nontheap}
	{\rm
	Consider  $0 \rightarrow \Z_3 \stackrel{\iota}{\rightarrow}  \Z_9  \stackrel{\pi }{\rightarrow} \Z_3 \rightarrow 0 $ as in Example~\ref{ex:heapextmodn} and the corresponding $\eta$. The $2$-chain $\alpha:= (1,0,2) + (0,1,0) + (1,2,0)$ is easily seen to be a heap $2$-cycle and $\eta(\alpha) =1 \neq 0$. %Lemma~\ref{lem:cycocy} implies that 
	Hence $\eta$ is non-trivial. % not nullcohomologous to the trivial cochain. 
	Therefore, $H^2_{\rm PA}(\Z_3,\Z_3) \neq 0$.
}
\end{example}

\section{From Heap Cocycles to Group Cocycles and Back}\label{sec:gpheap}

%Denote the group chain complex by $(G, \partial)$ for a group $G$.We consider the trivial action case. 
The main purpose of this section is to elucidate connections between group (co)homology and heap (co)homology.
For a group $G$, let denote the group chain complex by $(G, \partial)$.   We consider the trivial action case.

\subsection{Group Homology and type 0 Heap Homology}
In this section we provide an explicit relation between type 0 heap homology and group homology. % added

%To illustrate the general case, we start with the following computation.
%We recall that the group $2$-cocycle condition with trivial action on the coefficient group is formulated as 
%$$
%\theta  ( x, y ) + \theta(xy , z) =  \theta (y , z) + \theta (x, y  z) .
%$$
%
%\begin{lemma}\label{2cocyheaptogp}
%Let $X$ be a heap, $e \in X$, and $X$ be endowed with the group structure associated with $e$.
%Let  $\eta$ be a heap $2$-cocycle of type $0$.
%Then $\theta (x, y) := \eta (x, e, y) $ is a group $2$-cocycle.
%\end{lemma}
%
%\begin{proof}
%One computes 
%\begin{eqnarray*}
%\lefteqn{  \theta ( x, y ) + \theta(xy , z) }\\
%&=&   \eta( x, e, y ) + \eta( [ x, e, y ] , e, z) \\
%&=&   \eta( y ,  e, z) + \eta(  x, e, [ y  , e, z ]  ) \\
%&=& \theta (y , z) + \theta (x, y  z) 
%\end{eqnarray*}
%as desired.
%\end{proof}
%
%
%\begin{remark}\label{rem:HtoG}
%{\rm
%This construction of 2-cocycles  corresponds to the construction of a group from a heap through extensions as follows.
%Let $X$ be a heap, $A$ an abelian group, $E=X \times A$ the heap extension defined in Lemma~\ref{def:heapext}
%with a heap 2-cocycle $\eta$. 
%Let $(e, c) \in E$ be a fixed element.
%Then the group structure on $E$ defined from the heap structure on $E$ is computed as
%\begin{eqnarray*}
%\lefteqn{ (x, a) \cdot (y,b) }\\
%&=& 
%[ (x, a), (e, c) , (y, b) ] 
%\\
%&=& 
%( \ [x,e,y], a-c+b + \eta( x, e, y) \ ) \\
%&=& 
%( x y, a+b + \theta (x, y) )
%\end{eqnarray*}
%giving rise to the relation $\theta(x,y)=\eta(x,e,y)-c$, 
%a difference of a constant comparing to Lemma~\ref{2cocyheaptogp}. 
%
%}
%\end{remark}

%The general case is formulated as follows.

\begin{proposition}\label{2cocyheaptogp}
Let $X$ be a heap, $e \in X$, and $G$ be  the associated group, so that $xy=[x,e,y]$ for all $x, y \in X$.
Let $\Psi_n: C^G_n \rightarrow C^{\rm (0)}_{n}(X) $ be the map on chain groups defined by 
$$\Psi_n (x_1,\ldots ,x_n) = (x_1,e,x_2, e, \ldots ,e,x_n) . $$
Then $\Psi_{\bullet}$ is a chain map and therefore induces a well defined map $$\bar{\Psi}_n: H^{(0)}_{n}(X) \rightarrow H^G_{n}(X).$$
\end{proposition}

\begin{proof}
%Let $\delta_{\rm G}$ and $\delta_{\rm H}$ denote the group and heap differential maps.
%	By  direct investigation we have on the one hand:
%\begin{eqnarray*}
%\lefteqn{ h^{n-1} %%% h^n changed to n-1
%\delta_{\rm H} f (x_1,\ldots , x_n)} \\
%&=& h^{n-1}  (\delta_{\rm H} f )  (x_1,\ldots , x_n) \\  %% added
%&=&   (\delta_{\rm H} f ) ( x_3,e, \ldots , e,x_n ) \\
%		&=& - f(x_2,e, \ldots , e,x_n)\\ 
%		&& + \sum_{i=1}^{n-1} (-1)^{i+1} 
%		f(x_1,e,\ldots , e, x_{i-1}, e, [x_{i}, e , x_{i+1}], e, x_{i+1}, e, \ldots , e, x_n) \\
%		% f(x_1,e,\ldots , [x_{2i-1},x_{2i},x_{2i+1}], x_{2i+2}, \ldots , e, x_n).
%	&&	+ sign (-1)^{n+1}  f(x_1, e, \ldots, e, x_{n-1}) . %%% added with  sign (-1)^{n+1} 
%		\end{eqnarray*}
%	 On the other hand we also have: 
%	\begin{eqnarray*}
%		\lefteqn{ \delta_{\rm G} h^{n} %%% h^{n-1} changed to n
%		f(x_1,\ldots ,x_n) }\\
%		&=&  \delta_{\rm G} (h^n f) (x_1,\ldots ,x_n) \\
%				&=& - f(x_2,e,x_3,e,\ldots , e, x_n)\\
%		&& + \sum_{i=1}^{n-1} (-1)^{i+1} f(x_1, e, \ldots, e, x_ix_{i+1}, e, \ldots , x_n) \\
%		& & +(-1)^{n+1}  f(x_1, e, \ldots, e, x_ix_{i+1}, e, \ldots , x_{n-1})  %%% added
%		\end{eqnarray*}
%%Since, in the first equality, $x_{2i-1} = x_i$ and $x_{2i+1}= x_{i+1}$	for all $i=1, 2, \ldots n$, making use of the definition of product as 
%Since $xy = [x,e,y]$  we see that the two terms coincide. Therefore $h$ is a cochain map.
This is a direct computation, using the fact that $[x_{2i},e,x_{2i+1}] = x_{2i}\cdot x_{2i+1}$ by definition. 
%The type zero heap differentials coincide therefore with the group differentials. 
\end{proof}
%\begin{corollary}
%	The map in Lemma \ref{2cocyheaptogp} induces a well defined map in cohomology.
%\end{corollary}
%\begin{proof}
%	Just observe that the assignment in Lemma \ref{2cocyheaptogp} corresponds precisely to the map $h^2$.
%\end{proof}

\begin{remark}\label{rem:HtoG}
	{\rm
		By dualizing Proposition~\ref{2cocyheaptogp}, we obtain a cochain map between type %zero 
		0 heap cohomology and group cohomology. 
In the specific case of the second cohomology group, we observe that 
Proposition~\ref{2cocyheaptogp} % the previous result 
corresponds to the construction of a group from a heap through extensions as follows.
Let $X$ be a heap, $A$ an abelian group, $E=X \times A$ the heap extension defined in Lemma~\ref{def:heapext}
with a heap 2-cocycle $\eta$. 
Let $(e, c) \in E$ be a fixed element.
Then the group structure on $E$ defined from the heap structure on $E$ is computed as
\begin{eqnarray*}
\lefteqn{ (x, a) \cdot (y,b) }\\
&=& 
[ (x, a), (e, c) , (y, b) ] 
\\
&=& 
( \ [x,e,y], a-c+b + \eta( x, e, y) \ ) \\
&=& 
( x y, a+b + \theta (x, y) )
\end{eqnarray*}
giving rise to the relation $\theta(x,y)=\eta(x,e,y)-c$, 
a difference of a constant comparing to Proposition \ref{2cocyheaptogp}.
}
\end{remark}

%Let us consider now a chain element $(x_1,\ldots  , x_{2n-1})\in C^{(0)}_n(X)$, for some $n\geq 2$, with $x_{2i} = e$ for all $i = 1, 2, \ldots  n-1$. It is easy to see that the type zero heap differential maps  $(x_1,\ldots  , x_{2n-1})$ to a sum of elements of type $(y_1,\ldots  , y_{2n-3})$ with the same property of having $y_{i} = e$, for each $i = 1, 2, \cdot n-2$. We can therefore define a subcomplex $\widehat{C}_{\bullet}^{(0)}(X)$, where $\widehat{C}_n^{(0)}(X)$ is by definition freely generated by $2n-1$ tuples of elements of $X$, with even entries being equal to $e$, for $n\geq2$, and $C_1^{(0)}(X)$, for $n=1$. 
\begin{sloppypar}
Let $X$ be a heap, and $e \in X$.
We define chain subgroups $\widehat{C}_n^{(0)}(X)$ to be the free abelian group generated by 
$$\{ ( (x_1,\ldots  , x_{2n-1})\in C^{(0)}_n(X) \ | \ x_{2i}=e, \ i = 1, 2, \ldots ,  n-1 \} , $$
and $\widehat{C}_1^{(0)}(X)={C}_1^{(0)}(X)$. 
It is checked by direct computation that $\partial^{(0)}_n ( \widehat{C}_n^{(0)}(X) ) \subset \widehat{C}_{n-1}^{(0)}(X)$,
so that $\{ \widehat{C}_n^{(0)}(X) , \partial^{(0)}_n \}$ forms a chain subcomplex.
Let $\widehat{H}_n^{(0)}(X) $ denote the homology groups of this subcomplex, and let 
$\widetilde{H}^{(0)}_n(X)$ denote the relative homology groups for the quotient $ C_n^{(0)}(X)/  \widehat{C}_n^{(0)}(X)$.
We now have the following result.
\end{sloppypar}

\begin{theorem}
	In the same setting as in Proposition~\ref{2cocyheaptogp}, the map $\bar{\Psi}_{n}$ is an injection for all $n$. 
	Furthermore, %Also, 
	there is a long exact sequence of homology groups
	$$
	\cdots  \rightarrow H^G_n(X) \rightarrow H^{(0)}_n(X) \rightarrow \widetilde{H}^{(0)}_n(X) \xrightarrow{\partial} H^G_{n-1}(X) \rightarrow \cdots .
	$$
%	where $\widetilde{H}^{(0)}_n(X)$ is the relative homology of $C^{(0)}_{\bullet}$ with respect to $\widehat{C}^{(0)}_{\bullet}(X)$.
\end{theorem}
\begin{proof}
	The chain map $\Psi_n$ gives an isomorphism between chain groups $C^G_n(X)$ and $\widehat{C}^{(0)}_n(X)$
	and commute with differentials, giving rise to an isomorphism of chain complexes.
	%Since the type zero and group differentials coincide, on chain elements of $\widehat{C}^{(0)}_n(X)$, it follows that 
	Through the map $\bar{\Psi}_n$,  $H^G_n(X)$ is identified with $\widehat{H}^{(0)}_n(X)$. 
	
	The second statement follows from the short exact sequence of chain complexes
	$$0 \rightarrow \widehat{C}^{(0)}_{\bullet}(X) \rightarrow C^{(0)}_{\bullet}(X) \rightarrow C^{(0)}_{\bullet}(X)/\widehat{C}^{(0)}_{\bullet}(X) \rightarrow 0$$
	using the isomorphism $\Psi_{\bullet}$ and defining $\partial$ as the usual connecting homomorphism, via the Snake Lemma.
\end{proof}

\begin{definition}
	{\rm
		The homology $\widetilde{H}^{(0)}_{\bullet}(X)$ is called the type zero {\it essential} heap homology. 
	}
\end{definition}

\begin{remark}
	{\rm
The essential homology of a group heap $X$ is regarded as a  measure of how far is group homology from being isomorphic to the type zero heap homology.	
}
\end{remark}
\begin{example}
	{\rm
	We show that $\widetilde{H}^{(0)}_{\bullet}(X)$ can be nontrivial. Consider the group heap correspoding to $\Z_2$. The $2$-chain $(0,1,1)$ is easily seen to be a type zero $2$-cycle. 
	%The class $[(0,1,1)]\in \widetilde{H}^{(0)}_2(X)$ can be seen to be nontrivial in the following way. 
	We show that the class $[(0,1,1)]\in \widetilde{H}^{(0)}_2(X)$ is nontrivial.
	The $2$-cochain $\eta(1,1,1) = \eta(1,1,0) = \eta(0,1,1) = 1$, and zero otherwise is a heap $2$-cocycle, as seen in Example~\ref{ex:H2heap}. As previously observed, a heap $2$-cocycle is also a type zero $2$-cocycle.
	 Furthermore, %since the differentials $\partial_1^{PA}$ and 
	 $\partial_1^{(0)}$ %coincide, 
	 is dual to $\delta^1_{\rm PA}$, so that 
	 $\eta$ is nontrivial as a type zero heap cocycle. Suppose that $[(0,1,1)] = 0$ in $\widetilde{H}^{(0)}_2(X)$. Then there is a $3$-chain $\alpha$ such that $\partial^{(0)}_3 \alpha - (0,1,1)\in \widehat{C}^{(0)}_2(X)$. Therefore $\eta(\partial^{(0)}\alpha - (0,1,1)) = 0$, since by definition $\eta$ vanishes %  is null 
	 on $\widehat{C}^{(0)}_2(X)$. Since $\eta(\partial^{(0)} ) = \delta^{(0)} \eta$ and $\eta$ is a type zero $2$-cocycle, we have obtained that $\eta(0,1,1) = 0$, in contradiction with the choice of $\eta$. Therefore $[( 0,1,1) ]$ is nontrivial in $\widetilde{H}^{(0)}_{2}(X)$.
}
\end{example}

\subsection{From Group Cocycles to PA Cocycles}

In this section we present a construction of PA 2-cocycles from group 2-cocycles.
The following gives an answer to a natural question on how the relation between groups and heaps descends to relations in  their homology theories.
It also provides a construction of ternary self-distributive 2-cocycles from group 2-cocycles through 
heap 2-cocycles (Section~\ref{sec:heapTSD}).
We recall that  the group 2-cocycle condition with trivial action on the coefficient group is written as $$\theta(x,y)+\theta(xy,z)=\theta(y,z)+\theta(x,yz)$$
for all $x,y,z \in G$ of a group $G$.
The {\it normalized} 2-cocycle satisfies
$\theta(x,1)=0=\theta(1,x)$, and it follows that normalized 2-cocycles satisfy $\theta(x,x^{-1})=\theta(x^{-1},x)$. 
%The notion of normalized cocycles easily generalizes to $n$-cocycles and they form a sub-complex of the group cochain complex. We indicate the normalized cohomology by the symbol $\widehat{H}^n_G(X)$.
Define the normalized 2-cochain group $\widehat{C}^2_G(X)$ to consist of normalized 2-cochains, and 
the normalized 1-cochain group $\widehat{C}^1_G(X)$ to consist of $f \in C^1_G(X)$ such that $f(1)=0$. 
Then these form a subcomplex up to dimension 2, and the corresponding 2-dimensional cohomology group is denoted by 
 $\widehat{H}^2_G(X)$.

\begin{theorem}\label{th:GtoH}
Let $G$ be a group, and $X$ be the associated heap defined by $[x, y, z]=x y^{-1} z $ for $x,y,z \in G$. 
Let $\theta$ be a normalized group $2$-cocycle with trivial action on the coefficient group $A$.
Then $$ \eta(x, y, z ) := \theta (x, y^{-1} ) + \theta ( xy^{-1}, z) -\theta (y, y^{-1}) $$ is a PA $2$-cocycle. 
This construction $\Phi_2( \theta ) = \eta $ 
defines a cohomology map $\bar{\Phi}_2 : \widehat{H}^2_G(X) \rightarrow H^2_{PA}(X)$.
\end{theorem}

\begin{proof}
	First we note that for an extension group $2$-cocycle $\theta$, the condition $y^{-1}(zu^{-1}) = ((uz^{-1})y)^{-1}$ implies the following identity
	\begin{eqnarray*}
		\lefteqn{\theta(z,u^{-1}) + \theta(y^{-1},zu^{-1}) - \theta(y,y^{-1}) - \theta(u,u^{-1})}\\
		&=& \theta(u,z^{-1}) + \theta (z,z^{-1}) + \theta(u^{-1},uz^{-1}) - \theta(y,y^{-1})\\
		&=& \theta (u,z^{-1}) + \theta (uz^{-1}, y) - \theta (z,z^{-1}) - \theta (uz^{-1}y,y^{-1}zu^{-1}),
	\end{eqnarray*}
which we call the {\it product-inversion} relation. Observe that the normalization condition has been implicitly used to rewrite the term corresponding to $\theta(y^{-1},y)$. 
For $\delta^2_{(1)} (\eta) =0$, one computes
\begin{eqnarray*}
	\lefteqn{\eta(x,y,z) + \eta([x,y,z],u,v)}\\
	&=& \underline{\theta(x,y^{-1})} + \underline{\theta(xy^{-1},z)} -\theta(y,y^{-1})\\
	&& + \theta (xy^{-1}z,u^{-1}) + \theta (xy^{-1}zu^{-1},v) - \theta(u,u^{-1})\\
	&=& \theta(y^{-1},z) + \underline{\theta(x,y^{-1}z)} - \theta(y,y^{-1}) \\
	&& + \underline{\theta(xy^{-1}z,u^{-1})} -  \theta(u,u^{-1}) + \theta(xy^{-1}zu^{-1},v)\\
	&=& \underline{\theta(y^{-1},z)} + \underline{\theta(y^{-1}z,u^{-1})} + \theta(x,y^{-1}zu^{-1})\\
	&& - \theta(y,y^{-1}) - \theta(u,u^{-1}) + \theta(xy^{-1}zu^{-1},v)\\
	&=& \theta(z,u^{-1}) + \theta(y^{-1},zu^{-1}) + \theta(x,y^{-1}zu^{-1})\\
	&& + \theta(xy^{-1}zu^{-1},v) - \theta(y,y^{-1}) -\theta(u,u^{-1})\\
	&=& \theta(u,z^{-1}) + \theta(uz^{-1},y) - \theta(z,z^{-1}) \\
	&& - \theta(uz^{-1}y,y^{-1}zu^{-1}) + \theta(x,y^{-1}zu^{-1}) + \theta(xy^{-1}zu^{-1},v)\\
	&=& \eta(u,z,y) + \eta(x,[u,z,y],v),
	\end{eqnarray*}
where we have underlined the terms undergoing the group $2$-cocycle relation at each step, and used the product-inverse relation in the penultimate equality. Similar computations show $\delta^2_{(2)} (\eta) =0$.

To complete the proof, consider the maps $\Phi_1 := - \mathbbm{1} : \widehat{C}^1_G(X) \rightarrow C^1_{PA}(X)$, and $\Phi_2: \widehat{C}^2_G(X) \longrightarrow C^2_{PA}(X)$, $\theta \mapsto \eta$, as in the previous part of the proof. It is easy to see that $\delta_{G}^1\Phi_2 = \Phi_1\delta_{PA}^1$, therefore showing that $\bar{\Phi}$ is well defined on cohomology groups.
\end{proof}

\begin{remark}\label{rem:GtoHext}
{\rm
Extensions of groups and heaps, in this case, are related as in  Remark~\ref{rem:HtoG}. The group extension is defined, for a group $G$ and the coefficient abelian group $A$,  by 
$$
(x, a) \cdot (y, b) = (xy, a + b + \theta(x,y) )
$$
for $x, y \in G$ and $a, b \in A$. 
For the heap $E=G \times A$ constructed from the group $E=G \times A$ defined above, 
one computes 
\begin{eqnarray*}
\lefteqn{  [ \ (x,a), (y, b) , (z, c) \ ] } \\
&=& 
(x, a) (y, b)^{-1} (z, c) \\
&=& 
(x, a) (y^{-1}, -b - \theta(y, y^{-1} )) (z,c) \\
&=& 
(x y^{-1} z , a-b+c + \theta(x, y^{-1} ) + \theta(xy^{-1}, z) - \theta(y, y^{-1} ) ) 
\end{eqnarray*}
so that we obtain the correspondence
$$
\eta(x,y,z)=\theta(x, y^{-1} ) + \theta(xy^{-1}, z) - \theta(y, y^{-1} ).
$$
}
\end{remark}

%%% The following commented out for now to avoid lack of focus

%\begin{remark}
%{\rm 
%Let $G$ be a group, and $X$ be the associated heap. 
Let $\theta$ be a group $2$-cocycle satisfying the {\em inverse property}: $\theta (x^{-1} , y^{-1} ) = -  \theta (y, x )$ for all $x, y \in X$. Then $\eta (x, y, z ):= \theta (x, y^{-1} ) + \theta ( xy^{-1}, z) $ is a PA $2$-cocycle.

\section{Ternary Self-Distributive Cohomology with Heap Coefficients}\label{sec:TSD}

In this section we introduce  a cohomology theory of ternary self-distributive operations with abelian group heap coefficients, 
and investigate extension theory by 2-cocycles.
A ternary operation $T$ on a set $X$ is called {\it ternary self-distributive} (TSD for short) if it satisfies
$$ T( T(x,y,z), u,v) = T(T(x,u,v), T(y,u,v), T(z,u,v)) $$
for all $x,y,z, u,v \in X$. 
Such operations have been widely studied (e.g., 
\cite{CEGM,ESZ} and references therein). 
The set $X$ with TSD operation $T$, or the pair $(X, T)$, is also called a {\it ternary shelf}. 
In \cite{CEGM,ESZ} homology theories of ternary shelves are defined and studied.
The theory introduced here
differs in the use of heap structures.

\begin{definition}\label{gen-diff-heap-coeff} 
{\rm
Let $(X, T)$ be a ternary shelf.
	We define the $n^{\rm th}$ chain group of $X$ with heap coefficients in $A$, denoted by $C_n^{\rm SD}(X)$, to be the free abelian group on $(2n-1)$-tuples $X^{2n-1}$.
We introduce differentials $\partial_n: C_n^{\rm SD}(X) \longrightarrow C_{n-1}^{\rm SD}(X)$ by the formula
	\begin{eqnarray*}\label{diff-heap-coeff}
	\lefteqn{\partial_n  (x_1,x_2,\ldots, x_{2n-2}, x_{2n-1})}\\
	&=& \sum_{i=2}^{n}(-1)^{n+i}[  (x_1,\ldots , \widehat{x_{2i}},\widehat{x_{2i+1}},\ldots, x_{2n-1}) \\
	&&-  (T(x_1,x_{2i},x_{2i+1}), \ldots , T(x_{2i-1}, x_{2i}, x_{2i+1}), \widehat{x_{2i}},\widehat{x_{2i+1}},\ldots, x_{2n-1})]\\
	&& + (-1)^{n+1}[(x_1,x_4,\ldots, x_{2n-1})- (x_2,x_4,\ldots , x_{2n-1})\\
	& & + (x_3,x_4,\ldots , x_{2n-1}) - (T(x_1,x_2,x_3),x_4,\ldots , x_{2n-1})],
	\end{eqnarray*}
	where  \ $\widehat{ }$ \ denotes the deletion of that factor.
	}
\end{definition}

\begin{proposition}\label{pro:diff-heap}
	The differential maps in Definition \ref{gen-diff-heap-coeff} satisfy the condition $\partial^2 = 0$. 
	%In particular we recover the result in Proposition \ref{delta2=0}.
	\end{proposition}
	\begin{proof}
		We can write the differential in the following form:
		$$
		\partial_n = (-1)^n\sum_{i=1}^{n}(-1)^{i}\partial_n^i,
		$$
		where $\partial_n^i$ is defined, for $i\geq 2$, by the formula 
		\begin{eqnarray*}
		\lefteqn{\partial_n^i (x_1,\ldots , x_{2n-1})} \\
		&=& (x_1,\ldots , \widehat{x_{2i}},\widehat{x_{2i+1}},\ldots, x_{2n-1}) \\
		&&- (T(x_1,x_{2i},x_{2i+1}), \ldots ,  T(x_{2i-1}, x_{2i}, x_{2i+1}), \widehat{x_{2i}},\widehat{x_{2i+1}},\ldots, x_{2n-1}),
			\end{eqnarray*}
		and, for $i = 1$, is defined by the formula
		\begin{eqnarray*}
				\lefteqn{\partial_n^i (x_1,\ldots , x_{2n-1})}\\
				&=& (x_1,x_4,\ldots, x_{2n-1})- (x_2,x_4,\ldots , x_{2n-1})\\
				& & + (x_3,x_4,\ldots , x_{2n-1}) - (T(x_1,x_2,x_3),x_4,\ldots , x_{2n-1}).
			\end{eqnarray*}	
		Now it remains to prove the relations $\partial_{n-1}^{j-1}\partial_n^i = \partial_{n-1}^i\partial_n^{j}$ for $i < j$. The cases with $i \geq 2$ are standard, while
%		the same as in \cite{ESZ}. 
		the remaining cases $1 = i < j$  can be checked directly.
%		We show, as an example, the computations for the case $i=1, j=2$. On the one hand we have
%		\begin{eqnarray*}
%		\lefteqn{\partial_{n-1}^1\partial_n^2  (x_1,\ldots , x_{2n-1})}\\
%		&=& (x_1,x_6,\ldots ,x_{2n-1})- (x_2,x_6, \ldots , x_{2n-1})\\
%		&& + (x_3,x_6, \ldots , x_{2n-1}) - (T(x_1,x_2,x_3), x_6, \ldots , x_{2n-1})\\
%		&& - (T(x_1,x_4,x_5), x_6, \ldots , x_{2n-1}) + (T(x_2,x_4,x_5), x_6, \ldots , x_{2n-1})\\
%		&& \hspace{1 in} - (T(x_3,x_4,x_5), x_6, \ldots , x_{2n-1})\\ 
%		&& +  (T(T(x_1,x_4,x_5),T(x_2,x_4,x_5),T(x_3,x_4,x_5)), x_6, \ldots,  x_{2n-1}).
%		\end{eqnarray*}
%	On the other hand we have
%	\begin{eqnarray*}
%		\lefteqn{\partial_{n-1}^1 \partial_n^1  (x_1,\ldots , x_{2n-1})}\\
%		&=& (x_1,x_6,\ldots ,x_{2n-1}) - (x_4,x_6, \ldots , x_{2n-1})\\
%		&& + (x_5,x_6, \ldots , x_{2n-1}) - (T(x_1,x_4,x_5), x_6, \ldots , x_{2n-1})\\
%		&& - (x_2,x_6,\ldots ,x_{2n-1}) + (x_4,x_6,\ldots ,x_{2n-1})\\
%		&& - (x_5,x_6,\ldots ,x_{2n-1}) + (T(x_2,x_4,x_5), x_6, \ldots , x_{2n-1})\\
%		&& + (x_3,x_6,\ldots ,x_{2n-1}) - (x_4,x_6,\ldots ,x_{2n-1})\\
%		&& + (x_5,x_6,\ldots ,x_{2n-1}) - (T(x_3,x_4,x_5), x_6, \ldots , x_{2n-1})\\
%		&& - (T(x_1,x_2,x_3), x_6, \ldots , x_{2n-1}) + (x_4,x_6,\ldots ,x_{2n-1})\\
%		&& - (x_5,x_6,\ldots ,x_{2n-1}) + (T(T(x_1,x_2,x_3),x_4,x_5), x_6, \ldots , x_{2n-1}).
%		\end{eqnarray*}
%	The two expressions coincide.
	\end{proof}

The cycle, boundary and homology groups and their duals with respect to an abelian group $A$ are defined as usual.
Our focus is on significance and constructions of 2-cocycles in relation to heaps for this theory, 
so that we provide explicit  cocycle conditions in low dimensions below.

\begin{example}\label{tsd-cocy-cond}
{\rm
Let $(X,T)$ be a ternary shelf, and $A$ be an abelian group heap.
Then  cochain groups and differentials dual to Definition~\ref{gen-diff-heap-coeff} 
in low dimensions are formulated as follows.
The cochain groups $C^n_{\rm SD}(X,A)$ are defined to be the abelian groups
of  functions $\{ f: X^{2n-1} \rightarrow A\}$. 
The differentials $\delta^n=\delta^n_{\rm SD}: C^n_{\rm SD}(X,A)\rightarrow C^{n+1}_{\rm SD}(X,A)$
are formulated 
for $n=1,2,3$ as follows.
\begin{eqnarray*}
\lefteqn{  \delta^1 \xi (x_1, x_2, x_3)  } \\
&=&   \xi(x_1) -  \xi(x_2) + \xi(x_3) - \xi (T(x_1, x_2, x_3)) , \\
\lefteqn{ \delta ^2\eta (x_1, x_2, x_3, x_4, x_5) }\\
&=& \eta (x_1, x_2, x_3) + \eta( T(x_1, x_2, x_3),  x_4, x_5 ) \\
& & -  \eta (x_1, x_4, x_5)  +   \eta (x_2, x_4, x_5)-  \eta (x_3, x_4, x_5) \\
& & -   \eta (T(x_1, x_4, x_5), T(x_2, x_4, x_5) , T(x_3, x_4, x_5)) , \\
\lefteqn{ \delta^3 \psi (x_1, x_2, x_3, x_4, x_5, x_6, x_7) }\\
&=& \psi (x_1, x_2, x_3, x_4, x_5)
+ \psi ( T(x_1, x_4, x_5),T( x_2, x_4, x_5), T(x_3, x_4, x_5) , x_6, x_7) \\
& & + \psi (   x_1, x_4, x_5, x_6, x_7) - \psi ( x_2, x_4, x_5,  x_6, x_7  ) +  \psi (   x_3,  x_4, x_5, x_6, x_7) \\
& & - \psi ( T(x_1, x_2, x_3) , x_4, x_5, x_6, x_7)  - \psi(x_1, x_2, x_3, x_6, x_7) \\
& & - \psi( T(x_1,x_6,x_7),  T(x_2,x_6,x_7), T(x_3,x_6,x_7), T(x_4,x_6,x_7), T(x_5,x_6,x_7) ) .
\end{eqnarray*}
The case $n=0$ is defined by convention that  $C^0_{\rm SD}(X,A)=0$.
}
\end{example}

%
%\begin{proposition}\label{delta2=0}
%The composites $\delta^2\delta^1$ and $\delta^3\delta^2$ are zero maps.
%\end{proposition}

%\begin{proof}
%A direct computation with the use of the equality 
%\begin{eqnarray*}
%\lefteqn{ T(T(x_1,x_2,x_3),x_4,x_5)}\\
%& = & T(T(x_1,x_4,x_5), T(x_2,x_4,x_5), T(x_3,x_4,x_5))
%\end{eqnarray*}
% implies that $\delta^2\delta^1=0$. A long computation also shows, making use of ternary self-distributivity, that $\delta^3\delta^2= 0$ as well. We will generalize the previous definition and show the result more generally below. 
%\end{proof}

%Using Definition \ref{tsd-cocy-cond} and Proposition \ref{delta2=0} we introduce the following:
%\begin{definition}
%	{\rm
%		We define the $n$-th cohomology group of $X$  with heap coefficients in $A$, written $H_{\rm SD}^n(X,A)$, to be the quotient
%		$$
%		{\rm ker} (\delta^n)/ {\rm im} (\delta^{n-1}).
%		$$
% 	}
%\end{definition}

We investigate properties of % and 
TSD 2-cocycles. We start with extensions by abelian group heaps.

\begin{definition}\label{tsd-ext}
	{\rm
		Let $(X,T)$ be a ternary self-distributive set, $A$ an abelian group heap and $\eta: X\times X\times X \longrightarrow A$ a $2$-cocycle of $X$ with values in $A$. We define the self-distributive cocycle extension of $X$ with heap coefficients in $A$, by the cocycle $\eta$ to be the cartesian product $X\times A$, endowed with the ternary operation $T'$ given by
		$$
		(x,a)\times (y,b)\times (z,c) \mapsto (T(x,y,z), a-b+c+ \eta(x,y,z)).
		$$	
		In this situation we denote the extension by $X\times_{\eta} A$.
}
\end{definition}

\begin{lemma}
The TSD $2$-cocycle condition gives extension cocycles of TSDs with abelian group heap coefficients.
Specifically, the ternary operation in Definition \ref{tsd-ext}, corresponding to a $2$-cocycle $\eta$ satisfying the second condition $\delta^2\eta=0$ in \ref{tsd-cocy-cond}, is ternary self-distributive.
\end{lemma}
%\begin{proof}
%The left hand side of ternary self-distributivity reads as follows
%\begin{eqnarray*}
%	\lefteqn{T'(T'((x,a),(y,b),(z,c)),(u,d),(v,e))}\\
%	&=& (T(T(x,y,z),u,v), a-b+c+\eta(x,y,z) -d+e+\eta(T(x,y,z),u,v)),
%	\end{eqnarray*}	
%while the right hand side is
%\begin{eqnarray*}
%	\lefteqn{T'(T'((x,a),(u,d),(v,e)),T'((y,b),(u,d),(v,e)),T'((z,c),(u,d),(v,e)))}\\
%	&=&   (T(T(x,u,v),T(y,u,v),T(z,u,v)), a-d+e+ \eta(z,u,v) \\
%	&& -[b-d+e +\eta(y,u,v)]+ c-d+e+\eta(z,u,v)\\
%	&& +\eta(T(x,u,v),T(y,u,v),T(z,u,v))).
%	\end{eqnarray*}
%The two terms are seen to be equal in virtue of ternary self-distributivity of $T$ and the $2$-cocycle condition applied to $\eta$.
%\end{proof}
\begin{definition}
	{\rm
 Given two extensions $X\times_{\eta} A$ and $X\times_{\eta'} A$, we define a morpshim of extensions to be a morphism of  ternary self-distributive sets making a commutative diagram identical to the one in Definition \ref{heapextmor}. An invertible morphism of extensions is called isomorphism.
}
\end{definition}

%\begin{remark}
	Similarly to Definition \ref{heapextmor}, %we have a subdivision in equivalence classes of extensions.
	this defines an equivalence relation and corresponding isomorphism classes. 
%\end{remark}
%Our next goal is to establish a bijective correspondence between second cohomology classes of a ternary self-distributive structure with coefficients in a heap, and heap extensions. 
We have the following result.
\begin{proposition}
	There is a bijective correspondence between $H_{\rm SD}^2(X,A)$ and equivalence classes of extensions.
\end{proposition}
\begin{proof}
	Similar to the group-theoretic case and Proposition~\ref{pro:heapext}.
%	The proof is completely analogous to the heap case. We show few details just to elucidate the procedure. For a given element $[\eta]\in H^2_{\rm SD}(X, A)$, define the corresponding class of extensions to be $X\times_{\eta} A$, according to Definition \ref{tsd-ext}. To show that the correspondence is well defined it is enough to show that if $[\eta] = [\eta']$, then $X\times_{\eta} A$ and $X\times_{\eta'} A$ are equivalent. This is easily seen by means of the map $\phi$ defined by the assignment $(x,a) \mapsto (x,a+f(x))$, where $f$ is such that $\eta - \eta' = \delta^1f$. The fact that this map is a morphism of ternary self-distributive sets is equivalent to the $2$-cocycle conditions. In fact on the one hand we have
%	\begin{eqnarray*}
%	\lefteqn{ \phi(T_{\eta}((x,a),(y,b),(z,c))) }\\
%	&=& (T(x,y,z), a-b+c+\eta(x,y,z) + f(T(x,y,z))),
%	\end{eqnarray*}
%	while on the other hand
%	\begin{eqnarray*}
%	\lefteqn{ 
%	T_{\eta'}(\phi(x,a),\phi(y,b),\phi(z,c)) }\\
%	&=&  (T(x,y,z), a+f(x) -b - f(y) +c + f(z) + \eta'(x,y,z) ).
% 	\end{eqnarray*}
% 	Our assertion now is readily seen to be true.
% 	We leave the rest of the details to the reader.
\end{proof}

\begin{example} {\rm 
Let $X=\Z_2$ with the TSD operation 
 $T(x,y,z)=x+y+z \in \Z_2$. 
 This is in fact the abelian heap $\Z_2$ and by  Lemma~\ref{lem:heapTSD} below, the same operation is self-distributive.
	In this example we compute the first cohomology group $H^1_{\rm SD}(X,\mathbb{Z}_2)$ and the second cohomology group $H^2_{\rm SD}(X,\mathbb{Z}_2)$
	with coefficients in the abelian heap $\mathbb{Z}_2$.  For a function $f: X \rightarrow \mathbb{Z}_3$, a straightforward computation gives that 
	$\delta^1(f)=0$. % is the zero map.  
	This gives $H^1_{\rm SD}(X,\mathbb{Z}_2)\cong C^1_{\rm SD}(X,\mathbb{Z}_2)$.  To compute the kernel of $\delta^2$, let us write an element $\phi:X^3 \rightarrow \mathbb{Z}_2$ in terms of characteristic functions as $\phi=\sum_{x,y,z}\phi(x,y,z)\chi_{(x,y,z)}$.  Then $\delta^2 (\phi)=0$ gives the following system of equations in $\mathbb{Z}_2$:
	$$
	\begin{cases}
	\phi(1, 1, 1)+\phi(0, 0, 0)=0\\
	\phi(1, 1, 0)+\phi(0, 0, 1)=0\\
	\phi(1, 0, 1)+\phi(0, 1, 0)=0\\
	\phi(1, 0, 0)+\phi(0, 1, 1)=0
	\end{cases}$$
	implying that ${\rm ker}(\delta^2)$ is 4-dimensional with a basis $\chi_{(1,1,1)}+ \chi_{(0,0,0)},\;$ $\chi_{(1,1,0)}+ \chi_{(0,0,1)},\;$ $ \chi_{(1,0,1)} + \chi_{(0,1,0)},\;$ and $\chi_{(1,0,0)}+ \chi_{(0,1,1)}$.  Since ${\rm im}(\delta^1)=0$  we then obtain that $H^2_{\rm SD}(X,\mathbb{Z}_2)\cong {\mathbb{Z}_2}^{\oplus 4} $.
}
\end{example}

\begin{example}
{\rm
	In this example we compute the first cohomology group $H^1_{\rm SD}(X,\mathbb{Z}_3)$ and the second cohomology group $H^2_{\rm SD}(X,\mathbb{Z}_3)$ for the same $X=\Z_2$ as above,
with coefficients in the abelian heap $\mathbb{Z}_3$.  For a function $f: X \rightarrow \mathbb{Z}_3$, a direct computation gives that 
$\delta^1(f)(1,0,1)=f(0)-f(1)$, $\delta^1(f)(0,1,0)=f(1)-f(0)$ and all other unspecified values of $\delta^1(f)(x,y,z)$ are zeros.  This gives $H^1_{\rm SD}(X,\mathbb{Z}_3)\cong \mathbb{Z}_3$.  
%To compute the kernel of $\delta^2$, let us write an element $\phi:X^3 \rightarrow \mathbb{Z}_3$ in term of characteristic functions as $\phi=\sum_{x,y,z}\phi(x,y,z)\chi_{(x,y,z)}$.  
We continue to use the characteristic function notation. 
Then hand computations %, facilitated by Maple software, 
give that ${\rm ker}(\delta^2)$ is 3-dimensional with a basis $\chi_{(1,1,1)}+ \chi_{(0,0,0)}+ \chi_{(1,0,0)}+ \chi_{(0,1,1)},\;$ $ \chi_{(1,1,0)} + \chi_{(0,0,1)} -\chi_{(0,1,0)}$ and $\chi_{(1,0,1)}- \chi_{(0,1,0)}$.  Since ${\rm im}(\delta^1)$ is generated by $\chi_{(1,0,1)}- \chi_{(0,1,0)}$, we then obtain that $H^2_{\rm SD}(X,\mathbb{Z}_3)\cong \mathbb{Z}_3 \oplus \mathbb{Z}_3$.
}
\end{example}

\begin{example}
{\rm 
	If the TSD set $X$ is trivial, that is $T(x,y,z)=x$ for all $x,y,z \in X$, then the differentials $\delta^1$ and $\delta^2$ take the following simpler  forms:
	\begin{eqnarray*}
	  \delta^1 \xi (x,y,z) &=&   \xi(z) -  \xi(y), \\	
	\delta ^2\eta (x,y,z,u,v) &=& 
	\eta (y,u,v) -  \eta (z,u,v).
	\end{eqnarray*}
This gives, for an abelian group $A$, 
$${\rm im}(\delta^1)=\{\eta:X^3 \rightarrow A, \eta(x,y,z)=\xi(z)-\xi(y), \text{for some map}\; \xi:X \rightarrow A \}.$$
	Thus $H^1_{\rm SD}(X,A)=Z^1_{\rm SD}(X,A)$ is  the group of constant functions, which is isomorphic to $A$.
	The kernel of $\delta^2$ is given by
	$$
{\rm	ker}(\delta^2)=\{\eta:X^3 \rightarrow A, \eta(x,y,z)=\eta(x', y,z), \forall x,x',y,z \in X \}
	$$
	that are functions constant on the first variable.
	Hence $Z^2_{\rm SD}(X,A)$ is isomorphic to  $ A^{X \times X}$,  the group of functions $A^{X \times X}$ from $X \times X $ to $A$. This group has the subgroup
	$B^1_{\rm SD}(X,A)={\rm im}(\delta^1)=\{ \eta(x, y, z)=\xi(z) - \xi(y) \ | \ \xi \in A^X\}$.
%	and the quotient group is  $H^2(X,A)$.

%%% add:
For example, if $X$ is an $n$ element set and $A=\Z_p$ for a prime $p$, then 
$Z^2_{\rm SD}(X,A)\cong \Z_p^{\oplus n^2}$, $B^1_{\rm SD}(X,A)\cong \Z_p^{\oplus n}$ and $H^1_{\rm SD}(X,A)\cong \Z_p^{\oplus n}$.
%%% check
	}
\end{example}

%% added:
The following provides an algebraic meaning of the TSD 3-cocycle condition.

\begin{proposition}
The TSD $3$-cocycle condition gives obstruction cocycles of TSDs for short exact sequences of coefficients.
Specifically, let $X$ be a TSD set and consider a short exact sequence of abelian groups,
$$0 \longrightarrow H \stackrel{\iota}{\longrightarrow} E  \stackrel{\pi}{\longrightarrow}  A\longrightarrow 0 , $$ %maps added
where $E$ is the extension heap corresponding to the $2$-cocycle $\phi\in Z^2(X,A)$, and a section $s: A\longrightarrow E$, such that $s(0)=0$, %%% mapping zero to zero,
 the obstruction for %%%$s$ to be a homomorhism
 $s\phi$ to satisfy the $2$-cocycle condition %%% rewording 
  is a $3$-cocycle with heap coefficients in $H$.
\end{proposition}
\begin{proof}
We construct the mapping $\alpha: X^5 \longrightarrow H $ 
by the equality %%% as follows
\begin{eqnarray*}
\lefteqn{ \iota \alpha (x_1,\ldots , x_5)} \\
&=& s\phi (x_1,x_2,x_3) - s\phi(T(x_1,x_4,x_5),T(x_2,x_4,x_5),T(x_3,x_4,x_5))\\
&& + s\phi(T(x_1,x_2,x_3), x_4, x_5 )  % x_4, x_5 ) added
 - s\phi(x_1,x_4,x_5)\\
&& + s\phi(x_2,x_4,x_5) - s\phi(x_3,x_4,x_5).
\end{eqnarray*}	
Since $\phi$ satisfies the $2$-cocycle condition, we see that $\pi\alpha$ is the zero map, where $\pi: E \rightarrow A$ is the projection. It follows that there is 
%%% $\alpha$ factors through the inclusion $H\rightarrow E$ and can therefore be considered as a map $X^5\longrightarrow H$. 
 $\alpha: X^5 \longrightarrow H $ satisfying the above equality. 
 It is proved that  % We claim that 
 $\alpha:X^5\longrightarrow H$ so defined satisfies the $3$-cocycle condition with heap coefficients in $H$,
 by a direct (though long) calculation. 
\end{proof}

\section{From Heap Cocycles to TSD Cocycles}\label{sec:heapTSD}

In this section, we show that heaps and their 2-cocycles give rise to those for TSDs.
Although a heap gives rise to a group, and $T(x,y,z)=xy^{-1}z$ gives a TSD operation, the following lemma 
provides a direct argument, which provides an idea for the proof of Theorem~\ref{thm:h2cocyTSD}.

\begin{lemma}\label{lem:heapTSD}
	%If $X$ is a heap, then it is a ternary shelf.
	If $[-]$ is a heap operation on $X$, then the same operation is ternary self-distributive.
\end{lemma}

\begin{proof}
	First we note that
	for a heap operation it holds that
	$$ [ [x, y, z], z,y]=  [x, y, [z, z, y] ] = [x, y, y]=x. $$
	Then one computes
	\begin{eqnarray*}
		\lefteqn{T ( T (x_1, x_4, x_5) ,  T (x_2, x_4, x_5) , T (x_3, x_4, x_5) ) }\\
		&=& [ \ [ x_1, x_4, x_5]  ,  [ x_2, x_4, x_5]  , [x_3, x_4, x_5] \ ] \\
		&=&  [ \  x_1 , [\   [ x_2, x_4, x_5] ,  x_5, x_4 ] , [x_3, x_4, x_5] \  ] \\
		&=&  [ \  x_1 ,    x_2 , [x_3, x_4, x_5] \  ] \\
		&=&   [ \  [ x_1 ,    x_2 , x_3] , x_4, x_5  ] \\
		&=& T ( T (x_1, x_2, x_3) ,   x_4, x_5) )
	\end{eqnarray*}
	as desired. The notation $T(x,y,z)=[x,y,z]$ was used for clarification. 
\end{proof}

\begin{theorem}\label{thm:h2cocyTSD}
Let $X$ be a heap, with the operation regarded as a TSD operation by Lemma~\ref{lem:heapTSD}, and let $A$ be an abelian group.
Suppose that $\eta \in Z^2_{\rm H}(X, A)$, that is, $\eta$  satisfies $\delta^2_{(1)} \eta = 0 = \delta^2_{(2)} \eta $
and the degeneracy condition.
% Then by lemma: $\eta(x_1, x_2, x_3)+\eta( [x_1, x_2, x_3], x_3, x_2)=0$..
Then $\eta$ is a TSD $2$-cocycle, $\eta \in Z^2_{\rm SD}(X,A)$. This assignment induces an injection of $H^2_H(X,A)$ into $H^2_{SD}(X,A)$. 
\end{theorem}

\begin{proof}
%\begin{sloppypar}
We note that $\delta^2_{(1)} \eta = 0 = \delta^2_{(2)} \eta $ also implies $\delta^2_{(0)} \eta = 0$,
and the equality $[ [ x,y,z],z,y]=x$ from the proof of Lemma~\ref{lem:heapTSD}.
One computes
%\end{sloppypar}
\begin{eqnarray*}
\lefteqn{  \underline{ \eta (x_1, x_4, x_5) }   -   \eta (x_2, x_4, x_5)+  \eta (x_3, x_4, x_5) } \\
& &   + \underline{   \eta (T(x_1, x_4, x_5), T(x_2, x_4, x_5) , T(x_3, x_4, x_5)) }\\
&=& 
- \eta ( [ x_2, x_4, x_5], x_5, x_4 ) 
+ \underline{ \eta ( x_1, [ [ x_2, x_4, x_5], x_5, x_4 ], [x_3, x_4, x_5] ) }  \\
& & \hspace{1.5in} (= \eta ( x_1, x_2, [x_3, x_4, x_5] ) \ ) \\
& &  -   \eta (x_2, x_4, x_5)+ \underline{   \eta (x_3, x_4, x_5) } \\
&=& 
 \eta (x_1, x_2, x_3)
+ \eta ( [x_1,  x_2, x_3],  x_4, , x_5 ) \\
& & - \underline{  \eta ( [ x_2, x_4, x_5 ], x_5, x_4 )} - \underline{\eta( x_2, x_4, x_5)} \\
& = & 
 \eta (x_1, x_2, x_3)
+ \eta ( T (x_1,  x_2, x_3),  x_4, , x_5 ) 
\end{eqnarray*}
as desired.
The equalities follow from $\delta^2_{(1)} \eta = 0$, $\delta^2_{(0)} \eta = 0$ and
Lemma~\ref{lem:type0eqn}, respectively, and the underlined terms indicate where they are applied.
This proves that we have an inclusion $h: Z^2_{\rm H} (X, A ) \hookrightarrow Z^2_{\rm SD}(X,A)$. Since we have the equality $C^1_{\rm H} (X, A ) =  C^1_{\rm SD}(X,A)$ and the first cochain differentials for heap and TSD cohomologies coincide up to sign,   $\delta^1_{\rm H}=- \delta^1_{\rm SD}$, we have $h (\delta^1_{\rm H}(f) ) = - \delta^1_{\rm SD}( h(f) )$ 
and   $h (  B^2_{\rm H} (X, A ) ) \subset B^2_{\rm SD}(X,A)$, so that $h$ induces a homomorphism
$\bar{h}: H^2_{\rm H} (X, A ) \rightarrow H^2_{\rm SD}(X,A)$. Lastly, the map $\bar{h}$ is injective. Indeed, for $\eta \in  Z^2_{\rm H}(X,A)$,
assume that  $ h(\eta) \in Z^2_{\rm SD}(X,A)$ is null-cohomologous.
Then 
$h(\eta) = \delta^1_{\rm SD}  (\xi') $  for some $\xi' \in C^1_{\rm SD}(X,A)$. 
For $\xi=-\xi' \in C^1_{\rm H}$, we have $\eta=\delta^1_{\rm H} (\xi)$, so that $\eta$ is 
null-cohomologous in $ Z^2_{\rm H}(X,A)$.
\end{proof}

%Let $X$ and $A$ be as in Theorem~\ref{thm:h2cocyTSD}
%By this theorem, the homomorphism 
%$h: Z^2_{\rm H} (X, A ) \rightarrow Z^2_{\rm SD}(X,A)$ is defined by $h(\eta)=\eta$. 
%The groups $C^1_{\rm H} (X, A )  $ and $C^1_{\rm SD}(X,A)$ are the same group of functions $\{ f: X \rightarrow A\} $. 
%Define the same map $h:  C^1_{\rm H} (X, A ) \rightarrow C^1_{\rm SD}(X,A)$ by $h(f)=f$.
%Since the first cochain differentials coincide up to sign,   $\delta^1_{\rm H}=- \delta^1_{\rm SD}$, 
%for the heap $\delta^1_{\rm H}$ and TSD $\delta^1_{\rm SD}$,
%we have $h (\delta^1_{\rm H}(f) ) = - \delta^1_{\rm SD}( h(f) )$ 
%and   $h (  B^2_{\rm H} (X, A ) ) \subset B^2_{\rm SD}(X,A)$. 
%Hence 
% $h$ induces a homomorphism
%$\bar{h}: H^2_{\rm H} (X, A ) \rightarrow H^2_{\rm SD}(X,A)$.
%We have the following.
%
%\begin{proposition}\label{prop:inj}
%The induced homomorphism $\bar{h}$ is injective.
%\end{proposition}
%
%\begin{proof}
%For $\eta \in  Z^2_{\rm H}(X,A)$,
%assume that  $ h(\eta) \in Z^2_{\rm SD}(X,A)$ is null-cohomologous.
%Then 
%$h(\eta) = \delta^1_{\rm SD}  (\xi') $  for some $\xi' \in C^1_{\rm SD}(X,A)$. 
%For $\xi=-\xi' \in C^1_{\rm H}$, we have $\eta=\delta^1_{\rm H} (\xi)$, so that $\eta$ is 
%null-cohomologous in $ Z^2_{\rm H}(X,A)$.
%\end{proof}

\begin{example}
	{\rm
		In Example~\ref{ex:nontheap},  a nontrivial heap $2$-cocycle $\eta$ was given for $X=\Z_3=A$. 
		By Theorem~\ref{thm:h2cocyTSD}, $\eta$ is a non-trivial TSD 2-cocycle. 
		Hence we obtain $H^2_{\rm SD}(X, A)\neq 0$.
	}
\end{example}

%\begin{example}
%{\rm 
%We examine the construction of the preceding theorem for $X=\Z_2$.
%Let $\theta = \chi_{(1,1)}\in Z^2_{G} (\Z_2, \Z_2)$ discussed in Example~\ref{ex:Z211}.
%The $2$-cocycle $\eta$ corresponding to $\theta$ can be seen to be a non-trivial TSD $2$-cocycle using the same cochain $\alpha := (0,1,1)$ as in Example~\ref{ex:Z211}, easily seen to be also a TSD $2$-cycle. The combination of Lemma~\ref{lem:GtoH} and Theorem~\ref{thm:h2cocyTSD} has produced a non-trivial TSD $2$-cocycle.
%}
%\end{example}

%\begin{example}
%	{\rm
%		In Example~\ref{ex:nontheap}, it has been given a nontrivial heap $2$-cocycle. We show now that, applying Theorem~\ref{thm:h2cocyTSD}, the same $\eta$ is a nontrivial TSD $2$-cocycle. To this end, we consider the $2$-chain $\beta := (1,0,2) + (0,1,0) - (1,2,0)$ that can be seen to be a TSD $2$-cycle by direct computation. Observe that it differs from the heap $2$-cycle in Example~\ref{ex:nontheap} only by a sign. An application of a TSD version of Lemma~\ref{lem:cycocy} shows that $\eta$ is a nontrivial TSD $2$-cocycle, since $\eta(\beta) \neq 0$.
%	}
%\end{example}

\begin{remark} 
{\rm
The construction in Theorem~\ref{thm:h2cocyTSD}
and taking extensions commute (c.f.  Remarks \ref{rem:HtoG} and \ref{rem:GtoHext}).
Indeed, for a heap $X$ and an abelian heap $A$, 
the heap extension $X \times A$ by a heap 2-cocycle $\eta$ is defined by 
$$ [ (x, a), (y, b), (z, c) ] = (\, [x,y,z], a-b+c+\eta(x,y,z) \, ), $$
and Lemma~\ref{lem:heapTSD}
states that this heap operation gives a ternary shelf. On the other hand, this is the extension
of a ternary shelf by a TSD 2-cocycle $\eta$ with the heap coefficient $A$
by Definition~\ref{tsd-ext}.
}
\end{remark}

\section{Internalization}\label{sec:cat}

In this section we generalize to monoidal categories, the construction of TSD
structures from heaps.
%we generalize the construction of TSD structures from heaps to monoidal categories.
%%add
Throughout the section all symmetric monoidal categories are strict (the associator 
$(A \boxtimes B) \boxtimes C \rightarrow A \boxtimes (B \boxtimes C) $, the right and left unitors
$I \boxtimes X \rightarrow X$ and $X \boxtimes I \rightarrow X$ are all identity maps, where $I$ is the unit object).

Let $(\mathcal{C}, \boxtimes)$ be a symmetric monoidal category, $(X,\Delta, \epsilon)$ be a comonoid object in $\mathcal{C}$ and consider a morphism $\mu: X\boxtimes X\boxtimes X \rightarrow X$. We translate the heap axioms of Section \ref{heaprev} into commutative diagrams in the category $\mathcal{C}$. 	The  equalities of type 1 and 2 
para-associativity %%% added
are defined by the commutative diagram
	$$
	\begin{tikzcd}
	X^{\boxtimes 3}\arrow[dr,swap,"\mu"]& X^{\boxtimes 5}\arrow[l,swap,"\mu\boxtimes\mathbbm{1}^2"]\arrow[r,"\mathbbm{1}^2\boxtimes\mu"]\arrow[d] & X^{\boxtimes 3}\arrow[dl,"\mu"] \\
	                        &             X           &
	\end{tikzcd}
	$$
where the central arrow corresponds to the morphism 
%$$\mu(\mathbbm{1}\boxtimes\mu\boxtimes\mathbbm{1})(\mathbbm{1}\boxtimes \tau\boxtimes\mathbbm{1}^2)(\mathbbm{1}^2\boxtimes \tau\boxtimes\mathbbm{1}). $$
$\mu(\mathbbm{1}\boxtimes\mu\boxtimes\mathbbm{1})\tau_{321}$
and $\tau_{321}$ is defined by 
$$\tau_{321}=(\mathbbm{1}\boxtimes \tau\boxtimes\mathbbm{1}^2)(\mathbbm{1}^2\boxtimes \tau\boxtimes\mathbbm{1})
(\mathbbm{1}\boxtimes \tau\boxtimes\mathbbm{1}^2). $$
The  type 0 para-associativity  is defined by 
$$
\begin{tikzcd}
X^{\boxtimes 5}\arrow[r,"\mu\boxtimes\mathbbm{1}^2"]\arrow[d,swap,"\mathbbm{1}^2\boxtimes\mu"]&  X^{\boxtimes 3}\arrow[d,"\mu"]\\
X^{\boxtimes 3}\arrow[r,swap,"\mu"] & X
\end{tikzcd}
	$$
and follows from those of types 1 and 2. 
The degeneracy conditions are formulated as commutativity of the following diagrams.
$$
\begin{tikzcd}
X\boxtimes X \arrow[r,"\mathbbm{1}\boxtimes \Delta"]\arrow[d,swap,"\mathbbm{1}\boxtimes \epsilon"] & X\boxtimes X\boxtimes X\arrow[dl,"\mu"]\\
X&
\end{tikzcd}
\text{and} \quad
\begin{tikzcd}
X\boxtimes X \arrow[r,"\Delta\boxtimes \mathbbm{1}"]\arrow[d,swap,"\epsilon\boxtimes \mathbbm{1}"] & X\boxtimes X\boxtimes X\arrow[dl,"\mu"]\\
X&
\end{tikzcd}
$$

\begin{definition}\label{def:heapob}
	{\rm
	A heap object in a symmetric monoidal category is a comonoid object $(X,\Delta, \epsilon)$,
	where $\epsilon: X \rightarrow I$ is a counital morphism to the unit object $I$, 
	 endowed with a morphism  of comonoids $\mu: X^{\boxtimes3} \rightarrow X$ making all the  diagrams above commute. 
}
\end{definition}

\begin{example}
	{\rm
	A (set-theoretic) heap in the sense of Section \ref{heaprev} is a heap object in the category of sets. 
}
\end{example}

The following %is a more sophisticated example. It has
 appeared  implicitly in \cite{ESZ}.
 
\begin{example}\label{ex:Hopf}
	{\rm
	Let $H$ be an involutory Hopf algebra (i.e. $S^2 = \mathbbm{1}$) over a field ${\mathbbm k}$. Then $H$ is a heap object in the monoidal category of vector spaces and tensor products, with the ternary operation $\mu$ induced by the assignment
	$$
  	x\otimes y\otimes z\mapsto xS(y)z
	$$ 
	for single tensors.
	Indeed, %on simple tensors, 
	we have
	\begin{eqnarray*}
    \lefteqn{\mu (\mu (x\otimes y\otimes z)\otimes u\otimes v)}\\
    &=& xS(y)zS(u)v \\
    &=& xS(y)S^2(z)S(u)v\\
    &=& x S(uS(z)y)v\\
    &=& \mu(x\otimes \mu (u\otimes z\otimes y)\otimes v)
    \end{eqnarray*}
corresponding to the commutativity of the diagram representing equality of type 1. Observe that we have used the involutory hypothesis to obtain the second equality. We also have 
\begin{eqnarray*}
	\lefteqn{  \mu (\mathbbm{1}\otimes \Delta)(x\otimes y)}\\
	&=& \mu ( x \otimes y^{(1)} \otimes y^{(2)}) \\
	&=& x S(  y^{(1)}  )  y^{(2)} \\
	&=& \epsilon (y) x
	\end{eqnarray*}
which shows the left degeneracy constraint. The rest of the axioms can be checked in a similar manner.
}
\end{example}

The opposite direction in the group-theoretic case is 
the assertion that a pointed heap generates a group by means of the operation $xy = [x,e,y]$.
The following is a Hopf algebra version
%of the group-theoretic result asserting that a pointed heap generates a group by means of the operation $xy = [x,e,y]$, 
and can be obtained by calculations.
More general statement of this can be found in \cite{BS} and below.

\begin{proposition}\label{prop:HtoHopf}
	Let $(X, [  -  ])$ be a heap object in a coalgebra category, and let 
	%$\eta: \mathbbm{k} \rightarrow X$ be a choice of a unit for the comonoidal structure on $X$.
	$e \in X$ be a group-like element (i.e., $\Delta (e) = e \otimes e$ and $\epsilon (e)=1$). 
	Then $X$ is an involutory Hopf algebra with 
	multiplication $m(x \otimes y):=\mu_e (x \otimes y):= [x \otimes e \otimes y]$, unit $e$, and antipode
	$S(x):=[ e \otimes x \otimes e ]$. 
\end{proposition} 

\begin{proof}
We use Sweedler's notation $\Delta(x)=x^{(1)} \otimes x^{(2)}$. 
The associativity of $m$ follows from the type 0 para-associativity of $\mu$.
A unit condition is computed by 
$$m (e \otimes x)= \mu (e \otimes e \otimes x)=\mu (\Delta (e) \otimes x)=x $$
by the degeneracy condition and the assumption that $e$ is group-like. The other condition  $m( x \otimes e)=x$ is similar.
The compatibility between $m$ and $\Delta $ is computed as 
\begin{eqnarray*}
 \Delta m (x \otimes y ) &=& \Delta \mu ( x \otimes e \otimes y ) \\
  & = & \mu \tau ( \Delta (x) \otimes  \Delta (e) \otimes  \Delta (y) ) \\
  &=&  \mu \tau ( x^{(1)} \otimes   x^{(2)} \otimes  e \otimes e  \otimes  y^{(1)} \otimes   y^{(2)}   ) \\
  &=& \mu (  x^{(1)} \otimes e  \otimes  y^{(1)} ) \otimes \mu (  x^{(2)} \otimes e  \otimes  y^{(2)} ) \\
  &=& m (  x^{(1)} \otimes   y^{(1)}  ) \otimes m (  x^{(2)} \otimes  y^{(2)} ) 
  \end{eqnarray*}
as desired, where $\tau$ is an appropriate permutation that give the third equality, and the group-like assumption 
is used in the second equality.
An antipode condition is computed as 
\begin{eqnarray*}
\lefteqn{  m ( S \otimes \mathbbm{1} )   \Delta (x) }\\
&=& \mu ( \mu ( e \otimes x^{(1)}  \otimes e ) \otimes e \otimes x^{(2)} ) \\
&=& \mu ( e \otimes x^{(1)}  \otimes  \mu ( e  \otimes e \otimes x^{(2)} ) ) \\
&=&  \mu ( e \otimes x^{(1)}  \otimes  \mu ( \Delta (e)  \otimes x^{(2)} ) ) \\
&=&  \mu ( e \otimes x^{(1)}  \otimes  x^{(2)} )  \epsilon (e) \\
&=& e \epsilon (x) 
\end{eqnarray*}
as desired, where the group-like condition $ \epsilon (e)$ and the degeneracy condition for $\mu$ were used. The other case
$ m ( \mathbbm{1} \otimes  S )   \Delta (x) =  e \epsilon (x) $ is similar.
This completes the proof.
\end{proof}

\begin{remark}
	{\rm
	Observe that $S$ so defined, is involutory. This observation corroborates the necessity of including the involutory hypothesis in Example~\ref{ex:Hopf}. 
}
\end{remark}

%\begin{remark}\label{rem:coaug}
%{\rm
%We observe a relation between a choice of a group-like element $e$ in Proposition~\ref{prop:HtoHopf}
%and a {\it coaugmentation map}  of a coalgebra.
%Let $(X, \Delta, \epsilon) $ be a coalgebra. A coaugmentation is a coalgebra morphism $\eta: \mathbbm{k} \rightarrow X$
%(i.e., $\Delta \eta = ( \eta \otimes \eta ) j$, where $j:  {\mathbbm k}  \rightarrow  {\mathbbm k} \otimes {\mathbbm k}$ is the canonical isomorphism,
%$j(1)=1 \otimes 1$) such that $\epsilon \eta = \mathbbm{1}|_\mathbbm{k}$. 
%Let $e=\eta(1)$. We show that $e$ is group-like.
%One computes $\Delta (e) = \Delta \eta (1) =  \eta (1) \otimes \eta (1) = e \otimes e$, and 
%$\epsilon (e) = \epsilon ( \eta (1)) =1$ as desired.
%We do not know, however, in general the converse holds, i.e., whether for any group-like element $e$, there exists a coaugmentation map $\eta$ such that $\eta (1)=e$. 
%We observe that an advantage of using coaugmentation map is the desired condition can be stated by a map, 
%without mention of particular elements, which become handy in categorical definitions as we see below.
%}
%\end{remark}
\begin{remark}\label{rem:coaug}
{\rm
	We observe a relation between a choice of a group-like element $e$ in Proposition~\ref{prop:HtoHopf}
	and a {\it coaugmentation map}  of a coalgebra.
	Let $(X, \Delta, \epsilon) $ be a coalgebra. A coaugmentation is a coalgebra morphism $\eta: \mathbbm{k} \rightarrow X$
	(i.e., $\Delta \eta = ( \eta \otimes \eta ) j$, where $j:  {\mathbbm k}  \rightarrow  {\mathbbm k} \otimes {\mathbbm k}$ is the canonical isomorphism,
	$j(1)=1 \otimes 1$) such that $\epsilon \eta = \mathbbm{1}|_\mathbbm{k}$. 
	Let $e=\eta(1)$. We show that $e$ is group-like.
	One computes $\Delta (e) = \Delta \eta (1) =  \eta (1) \otimes \eta (1) = e \otimes e$, and 
	$\epsilon (e) = \epsilon ( \eta (1)) =1$ as desired.
	%We do not know, however, in general the converse holds, i.e., whether for any group-like element $e$, there exists a coaugmentation map $\eta$ such that $\eta (1)=e$. 
	Conversely, for any group element $e\in X$, let $\eta$ be defined by $\eta(1)=e$.
	Then one computes 
	$\Delta \eta (1) = \Delta (e) =e \otimes e = \eta(e) \otimes \eta(e) = 
	( \eta \otimes \eta )j(1)$
	and  $\epsilon \eta (1) = \epsilon (e) = 1 $.
	We observe that an advantage of using coaugmentation map is that the desired condition can be stated by a map, 
	without mention of particular elements, which becomes advantageous in categorical definitions as we see below.
}
\end{remark}

We generalize Example~\ref{ex:Hopf} and Proposition~\ref{prop:HtoHopf}  to symmetric monoidal categories as follows,
using Remark~\ref{rem:coaug}.
For this purpose first we define a {\it coaugmentation} of a comonoid object $(X, \Delta, \epsilon )$ 
in a symmetric monoidal category with a unit object $I$ as a comonoidal morphism $\eta: I \rightarrow X$ 
such that $\epsilon \eta = \mathbbm{1}$. 

%We recall that the {\it product} in a category is an object  $X_1 \times X_2$ and morphisms $\pi_1: X_1 \times X_2 \rightarrow X_1$ and 
%$\pi_2: X_1 \times X_2 \rightarrow X_2 $ with the universal property that for any object $Y$ and 
%morphisms $f_i : Y \rightarrow X_i$, $i=1,2$, there exists a unique $f : Y \rightarrow X_1 \times X_2 $ such that 
%$f=\pi_1 f_i $ for $i=1,2$. For objects $X_1$ and $X_2$, if there is a product $X_1 \times X_2$, we call is a {\it product with projections}.

\begin{definition}
	{\rm
	Let $\mathcal C$ be a symmetric monoidal category. We define the {\it category of heap objects} in $\mathcal C$, $\mathcal{H}_{\mathcal{C}}$, as follows. The objects of $\mathcal{H}_{\mathcal{C}}$ are the heap objects as in Definition~\ref{def:heapob}. The morphisms are defined to be the morpshisms of $\mathcal{C}$ commuting with the heap maps and the comonoidal structures. A heap object $X$, is called {\it pointed}, if it is endowed with a 
	coaugmentation % unit 
	$\eta: I \rightarrow X$. %, with respect to the comonoidal structure. 
	The {\it category of pointed heap objects} in $\mathcal C$, $\mathcal{H}^*_{\mathcal{C}}$, is the category consisting of pointed heap objects over $\mathcal C$, and morphisms of heap objects commuting with the coaugmentations.
	
}
\end{definition}

%The dual of a  comonoid object $(X, \Delta, \epsilon)$ in a symmetric monoidal category ${\mathcal C}$  is 
%a monoidal object $(X, m, \eta)$ in ${\mathcal C}^{\rm op}$, where $\eta: I \rightarrow X$ is a unital morphism from
%the unit object $I$. 
%Let $iH_{\mathcal C}$ denote the category of involutory Hopf monoid objects  in ${\mathcal C}^{\rm op}$.
%In particular, an object $X$  in ${\mathcal C}^{\rm op}$ is equipped with a monoidal product $m$, a unit object $I$, 
%right and left unitors, comonoidal product $\Delta$, left and right counitors, and a morphism, an antipodal morphism $S$
%with $S^2=\mathbbm{1}$, 
%and a morphism $\eta: I \rightarrow X$ that satisfies $m (\eta \boxtimes \mathbbm{1})=\lambda$ and 
%$m (  \mathbbm{1} \boxtimes \eta)=\rho$. %%%% CHECK!

An  {\it involutory Hopf} monoid (object) in a symmetric monoidal category is equipped with 
a monoidal product $m$, a unit object $I$, 
a comonoidal product $\Delta$,  an antipode $S$ that is an  antimorphism 
(i.e., $\Delta S=(S \boxtimes S) \tau \Delta$ and $S m = m \tau (S \boxtimes S)$)
satisfying 
$
m (S \boxtimes \mathbbm{1}) \Delta = \eta \epsilon = m ( \mathbbm{1}\boxtimes  S )  \Delta$ and 
$ S^2 = \mathbbm{1}$, 
 a unit morphism $\eta: I \rightarrow X$ that satisfies the left and right unital conditions  $m (\eta \boxtimes \mathbbm{1})=\mathbbm{1} $ and 
$m (  \mathbbm{1} \boxtimes \eta)=\mathbbm{1} $,
and a counit morphism $\epsilon: X \rightarrow I$ that satisfies 
the left and right counital conditions $( \epsilon \boxtimes \mathbbm{1} ) \Delta =\mathbbm{1} $, 
$( \mathbbm{1} \boxtimes  \epsilon ) \Delta =\mathbbm{1} $, 
and $\epsilon \eta=\mathbbm{1}_I$. %% Do we need this ?? 

\begin{theorem}
	Let $\mathcal{C}$ be a symmetric monoidal category. There is an equivalence of categories between the category $\mathcal{H}^*_{\mathcal{C}}$ and  
	the category of involutory Hopf monoids in $\mathcal C$, $sH_{\mathcal C}$. 
	\end{theorem}	

	\begin{proof}
		We define a functor $\mathcal F: \mathcal{H}^*_{\mathcal{C}} \longrightarrow sH_{\mathcal C}$ as follows. Let $(X, \eta, \mu, \epsilon, \Delta)$ be a pointed heap object in $\mathcal C$,  define $\mathcal{F}(X) := (X, I, \eta, \lambda_\eta, \rho_\eta,  m, \epsilon, \Delta, S)$, where multiplication $ m := \mu \circ \mathbbm{1}\boxtimes \eta \boxtimes \mathbbm{1}$, antipode $S := \mu \circ \eta \boxtimes \mathbbm{1}\boxtimes \eta$, the unit object $I$, %%% <--- Check!
		the left and right unitors $\lambda_\eta$, $\rho_\eta$ by 
$\lambda_\eta := m (\eta \boxtimes \mathbbm{1} )  : I  \boxtimes X \rightarrow  X$ and 
$\rho_\eta := m (  \mathbbm{1} \boxtimes \mathbbm{1} ) :X \boxtimes I  \rightarrow X$,		
 and comonoidal structure unchanged. The functor $\mathcal F$ is defined to be the identity on morphisms.
 The fact that $\mathcal{F}(X)$ is a Hopf monoid in $\mathcal C$ is a translation of the computations in Example~\ref{ex:Hopf} in commutative diagrams or series of composite morphisms. 

Specifically, defining conditions are verified as follows.
The associativity of $m$ follows from the para-associativity as before.
The unit $\eta$ and the counit $\epsilon$ are unchanged and a left unital condition is checked by 
%The left unitor is 
\begin{eqnarray*}
\lefteqn{ 
%\lambda_\eta  = 
m (\eta  \boxtimes \mathbbm{1} ) = \mu (\eta  \boxtimes \eta  \boxtimes \mathbbm{1} ) }\\
& & 
= \mu ( \Delta \eta  \boxtimes \mathbbm{1} ) = ( \epsilon \boxtimes \mathbbm{1} )  (\eta \boxtimes  \mathbbm{1} )
= \epsilon \eta \boxtimes \mathbbm{1} =\mathbbm{1}_I \boxtimes \mathbbm{1} = \mathbbm{1}
\end{eqnarray*}
%which is an invertible morphism 
as desired. The right unital condition is computed similarly.
The compatibility between $m$ and $\Delta$ is computed as 
\begin{eqnarray*}
\lefteqn{ ( m \boxtimes m) %\tau 
( \Delta \boxtimes \Delta ) =
[    (\mu ( \mathbbm{1}\boxtimes  \eta \boxtimes \mathbbm{1} ) 
\boxtimes (\mu ( \mathbbm{1}\boxtimes  \eta \boxtimes \mathbbm{1} )  ] % \tau 
( \Delta \boxtimes \Delta ) } \\
& &
= (\mu \boxtimes \mu) ( \Delta \boxtimes  \Delta \eta \boxtimes \Delta )
=  (\mu \boxtimes \mu) ( \Delta \boxtimes  \Delta \boxtimes \Delta ) ( \mathbbm{1}\boxtimes  \eta \boxtimes \mathbbm{1} ) \\
& & 
=  \Delta \mu  (  \mathbbm{1}\boxtimes  \eta \boxtimes \mathbbm{1} ) 
=  \Delta m  .
\end{eqnarray*}
We note that the compositions involving  $ \Delta \boxtimes \Delta $ contain appropriate permutations of factors of objects.
The antipode condition is computed as 
\begin{eqnarray*}
\lefteqn{   m ( S \boxtimes \mathbbm{1} ) \Delta } \\
&=& 
\mu ( \mathbbm{1} \boxtimes \eta \boxtimes  \mathbbm{1})    ( \mu (\eta \boxtimes \mathbbm{1} \boxtimes \eta ) \boxtimes \mathbbm{1} )  \Delta \\
&=& 
\mu ( \eta \boxtimes \mathbbm{1} \boxtimes \mathbbm{1} ) ( \mu (\eta \boxtimes \eta \boxtimes \mathbbm{1} )   \boxtimes \mathbbm{1} )  \Delta \\
&=&
(\eta \boxtimes \epsilon) ( \epsilon \eta \boxtimes \mathbbm{1} ) \Delta \\
&=& \epsilon \eta .
\end{eqnarray*}
The other antipode condition is similar.
 
   Similarly, we define a functor $\mathcal G : sH_{\mathcal C} \longrightarrow \mathcal{H}^*_{\mathcal{C}}$ by the assignment on objects 
  $\mathcal G (X, \eta, m, \epsilon, \Delta, S) := (X, \eta, \mu, \epsilon, \Delta)$, 
% $\mathcal G (X, \lambda, \rho, m, \epsilon, \Delta, S, \eta) := (X, \eta, \mu, \epsilon, \Delta, \eta)$ 
with 
   $\mu := m  (m\boxtimes \mathbbm{1})(\mathbbm{1}\boxtimes S\boxtimes \mathbbm{1})$.   
   Also, $\mathcal G$ is the identity on morphisms. %We leave the details of the proof to the reader.
   %% add:
   The proof is obtained by sequences of equalities of composite morphisms mimicking the computations in 
   Example~\ref{ex:Hopf}. 
\end{proof}

Next we show that Lemma~\ref{lem:heapTSD} holds for a coalgebra (i.e. a comonoid in the category of vector spaces).
Although this is a special case of Theorem~\ref{thm:cat}, we include its statement and proof here to  illustrate and further motivate Theorem~\ref{thm:cat}.
For this goal, we slightly modify the definition of TSD maps in a symmetric monoidal categories, given in \cite{ESZ}.

\begin{definition}\label{def:TSD}
{\rm
Let $(X, \Delta, \epsilon)$ be a comoidal object in  a symmetric monoidal category $\mathcal{C}$.
A {\em ternary self-distributive object} $(X, \Delta, \epsilon, \mu)$ in $\mathcal{C}$  is a comonoidal object 
that satisfies the following condition:
$$
\mu (\mu^{\boxtimes 3})\shuffle_3 [ (\Delta' \boxtimes \mathbbm{1})\Delta \boxtimes (\Delta'\boxtimes \mathbbm{1})\Delta ]
= \mu (\mu \boxtimes \mathbbm{1}^{\boxtimes 2})
$$
where $\shuffle_3$ denotes the composition of switching maps corresponding to 
transpositions $(2\: 4) (3 \: 7) (6 \: 8)$ and $\Delta' := \tau \Delta$.
}
\end{definition}

This differs from the definition found in \cite{ESZ} only in the use of $\Delta'$ instead of $\Delta$. 
The main examples, set theoretical ones and  Hopf algebras,  satisfy both definitions.

\begin{proposition}\label{pro:hopfTSD}
	Let $H$ and $\mu$ be as in the Example \ref{ex:Hopf}.
	Then $\mu$ defines a ternary self-distributive object in the category of vector spaces. %, in the sense of \cite{ESZ}.
\end{proposition}

\begin{proof}
	One proceeds as in the proof of Lemma~\ref{lem:heapTSD}
	as follows.
	We use the Sweedler notation $\Delta (x)=x^{(1)} \otimes x^{(2)}$ and %by coassociativity, 
	$(\Delta \otimes {\mathbbm 1}) \Delta (x)=x^{(11)} \otimes x^{(12)}\otimes x^{(2)}$.
	Then one computes 
	\begin{eqnarray*}
		\lefteqn{ \mu( \mu (x\otimes y^{(1)} \otimes z^{(1)} ) \otimes z^{(2)} \otimes y^{(2)} ) } \\
		&=& \mu(x \otimes y^{(1)} \otimes  \mu (z^{(1)}  \otimes z^{(2)} \otimes y^{(2)} ) \\
		&=& \mu(x \otimes y^{(1)} \otimes  y^{(2)} ) \epsilon (z)  \quad = \quad x\  \epsilon (y) \epsilon (z)  .
	\end{eqnarray*}
	Then we obtain 
	\begin{eqnarray*}
		\lefteqn{ \mu  ( \mu (x_1\otimes x_4^{(12)} \otimes x_5^{(12)} )\otimes  \mu (x_2\otimes x_4^{(11)} \otimes x_5^{(11)}  )\otimes \mu (x_3\otimes x_4^{(2)}  \otimes x_5^{(2)} ) ) }  \\
		&=& \mu (  x_1 \otimes \mu (    \mu (  x_2\otimes x_4^{(11)} \otimes x_5^{(11)}  ) \otimes  x_5^{(12)} \otimes x_4^{(12)}  ) \otimes \mu ( x_3\otimes x_4^{(2)} \otimes x_5^{(2)} )   \ )  \\
		&=&  \mu (   x_1 \otimes    x_2 \ \epsilon (x_4^{(1)} ) \epsilon (x_5^{(1)} ) \otimes \mu ( x_3\otimes x_4^{(2)}\otimes x_5^{(2)} ) \  )  \\
		&=&   \mu (  \mu (  x_1 \otimes    x_2 \otimes x_3 )  \otimes \epsilon (x_4^{(1)} ) \epsilon (x_5^{(1)} ) ( x_4^{(2)}\otimes x_5^{(2)} )  )  \\
		&= & \mu  (  \mu  (x_1\otimes x_2\otimes x_3) \otimes   x_4\otimes x_5) )
	\end{eqnarray*}
	as desired. 
\end{proof}
Our goal, next, is to show that a more general version of Lemma~\ref{lem:heapTSD} and Proposition~\ref{pro:hopfTSD} holds in an arbitrary symmetric monoidal category. 
We first have the following preliminary result. 

\begin{lemma}\label{lem:catuv}
	Let $(X,\Delta,\epsilon, \mu)$ be a heap object in a symmetric monoidal category with tensor product $\boxtimes$ and %braiding (i.e. 
	switching morphism $\tau$. Then the following identity of morphisms holds
	$$
	\mu (\mu \boxtimes \mathbbm{1}^{\boxtimes 2})\tau_{4,5}\tau_{3,4}( \mathbbm{1\boxtimes}\Delta\boxtimes \Delta) = \mathbbm{1}\boxtimes \epsilon\boxtimes \epsilon.
	$$
\end{lemma}
\begin{proof}
	We observe that the following commutative diagram implies our statement.
	$$
	\begin{tikzcd}
	  X^{\boxtimes 5} \arrow[rr,"\tau_{4,5}\tau_{3,4}"]& & X^{\boxtimes 5}\arrow[rdd,bend left = 45,"\mathbbm{1}^2\boxtimes \mu"] \arrow[rrdd,bend left = 40,"\mu\boxtimes \mathbbm{1}^2"]& &  \\
	  X^{\boxtimes 4}\arrow[drrr,"\mathbbm{1}^3\boxtimes \epsilon"]\arrow[u,swap,"\mathbbm{1}^3\boxtimes \Delta"]\arrow[rr,"\tau_{3,4}"] & & X^{\boxtimes 4}\arrow[u,"\mathbbm{1}^2\boxtimes \Delta\boxtimes \mathbbm{1}"]\arrow[dr,"\mathbbm{1}^2\boxtimes\epsilon\boxtimes \mathbbm{1}"] & & \\
	  & & & X^{\boxtimes 3}\arrow[ddd,"\mu"]& X^{\boxtimes 3}\arrow[dddl,"\mu"]\\
	  & X^{\boxtimes 2}\arrow[ddrr,"\mathbbm{1}\boxtimes\epsilon"]\arrow[urr,"\mathbbm{1}\boxtimes \Delta"]& & &\\
	  X^{\boxtimes 3}\arrow[uuuu,bend left = 30,"\mathbbm{1}\boxtimes\Delta\boxtimes \Delta"]\arrow[drrr,"\mathbbm{1}\boxtimes\epsilon\boxtimes\epsilon"]\arrow[ur,"\mathbbm{1}^2\boxtimes\epsilon"]\arrow[uuu,swap,"\mathbbm{1}\boxtimes \Delta\boxtimes\mathbbm{1}"] & & & & \\
	  & & &  X&
		\end{tikzcd}
		$$
	where the rectangle on top, and the two triangles below commute because of naturality of the switching morphism, while the other parts of the diagram commute by heap and comonoid axioms. 
\end{proof}

\begin{theorem}\label{thm:cat}
	Let $(X,\Delta,\epsilon, \mu)$ be a heap object in a symmetric monoidal category $\mathcal{C}$. Then $(X,\Delta,\epsilon, \mu)$ is also a ternary self-distributive object in  $\mathcal{C}$.
\end{theorem}

\begin{proof}
	Since $(X,\Delta, \epsilon)$ is a comonoid in $\mathcal{C}$ by hypothesis, we just need to prove that ternary self-distributivity of $\mu$. We use the following commutative diagram
$$
\begin{tikzcd}
	&X^9\arrow[ddrr,"\mathbbm{1}^6\mu"]\arrow[r,"\tau"]& X^9\arrow[rr,"\mathbbm{1}^3\mu\mu"]& & X^5\arrow[dd,"\mu\mathbbm{1}^2"]\arrow[dddl,swap,"\tau\circ(\mathbbm{1}\mu\mathbbm{1})"]\\
	&&&&\\
	X^7\arrow[uur,"\mathbbm{1}^2\Delta'\Delta'\mathbbm{1}^3"]\arrow[r,"\mathbbm{1}^4\mu"]&X^5\arrow[drr,bend right = 10,swap,"\mathbbm{1}\epsilon\mathbbm{1}\epsilon\mathbbm{1}"]\arrow[r,"\mathbbm{1}^2\Delta'\Delta'\mathbbm{1}"]&X^7\arrow[r,"\tau"]&X^7\arrow[d,swap,"(\mathbbm{1}\mu\mathbbm{1})\circ(\mathbbm{1}\mu\mathbbm{1}^3)"]&X^3\arrow[dd,"\mu"]\\
	X^7\arrow[u,"\tau"]\arrow[r,"\mu\mathbbm{1}^4"]&X^5\arrow[rr,swap,"\mathbbm{1}\epsilon\mathbbm{1}\epsilon\mathbbm{1}"]&&X^3\arrow[dr,"\mu"]&\\
	X^5\arrow[rrru,"\mathbbm{1}^2\mu"]\arrow[u,"\mathbbm{1}^3\Delta\Delta"]\arrow[rr,"\mu\mathbbm{1}^2"]&&X^3\arrow[rr,"\mu"]&&X
\end{tikzcd}
$$
where we have omitted the symbol $\boxtimes$ in the product of morphisms, omitted the subscripts corresponding to the switching morphisms $\tau$, to slightly shorten the notation and, finally, we have used the notation $\circ$ to indicate the composition of morphisms. The leftmost map $\tau: X^7 \rightarrow X^7$ is the composition of symmetry constraints corresponding to the transposition $(5\: 6)(4\: 5)(5\: 6)(4\: 5)(3\: 4)$, proceeding clockwise, $\tau: X^9 \rightarrow X^9$ corresponds to $(3\: 4)(4\: 5)(2\: 3)$. The reader can easily find the correct compositions corresponding to the remaining $\tau$'s by a diagrammatic approach. The triangles on the right and at the bottom, are instances of type 1 and type 0 axioms, respectively. The middle triangle commutes as a consequence of Lemma \ref{lem:catuv}. The other diagrams can be seen to be commutative either by applying the comonoid axioms or naturality of the switching morphism. Finally, by direct inspection we  see that the upper perimeter of the diagram corresponds to the LHS of TSD, as stated in Definition \ref{def:TSD}. This completes the proof.
	\end{proof}

 \noindent
 {\bf Acknowledgment.}
 M.S. was  supported in part  by   NSF DMS-1800443.


\begin{thebibliography}{0}

	\bib{BH}{book}{
		author={Bergman, George M.},
		author={Hausknecht, Adam O.},
		title={Cogroups and co-rings in categories of associative rings},
	series={Mathematical Surveys and Monographs}
	volume = {45}
	DOI={http://dx.doi.org/10.1090/surv/045}, 
	publisher={American Mathematical Society, Providence, RI}
	}
\bib {BS}{article}{
	author={Booker, Thomas},
	author={Street, Ross},
	title = {Torsors, herds and flocks},
	journal = {Journal of Algebra},
	volume = {330},
	number = {1},
	pages = {346 - 374},
	year = {2011},
	issn = {0021-8693},
	doi = {https://doi.org/10.1016/j.jalgebra.2010.12.009},
	url = {http://www.sciencedirect.com/science/article/pii/S0021869310006393},
%	keywords = {Torsor, Herd, Flock, Monoidal category, Descent morphism},
%	abstract = "This paper presents non-commutative and structural notions of torsor. The two are related by the machinery of TannakaâKrein duality."
}
\bib{Bro}{book}{
	%			place={Cambridge}
	series={Graduate Texts in Mathematics}
	volume = {87}
	title={Cohomology of Groups}, 
	DOI={10.1007/978-1-4684-9327-6}, 
	publisher={Springer-Verlag New York
	}, 
	author={Brown, Kenneth S.}, year={1982}, 
}

	\bib{CCES}{article}{
			author={Carter, J. Scott},
			author={Crans, Alissa S.},
			author={Elhamdadi, Mohamed},
			author={Saito, Masahico},
			title={Cohomology of categorical self-distributivity},
			journal={J. Homotopy Relat. Struct.},
			volume={3},
			date={2008},
			number={1},
			pages={13--63},
			review={\MR{2395367}},
		}
		

	\bib{CEGS}{article}{
		author={Carter, J. Scott},
		author={Elhamdadi, Mohamed},
		author={Gra\~na, Matias},
		author={Saito, Masahico},
		title={Cocycle knot invariants from quandle modules and generalized
			quandle homology},
		journal={Osaka J. Math.},
		volume={42},
		date={2005},
		number={3},
		pages={499--541},
		issn={0030-6126},
		review={\MR{2166720}},
	}
	\bib{CES}{article}{
	author={Carter, J. Scott},
	author={Elhamdadi, Mohamed},
	author={Saito, Masahico},
	title={Twisted quandle homology theory and cocycle knot invariants},
	journal={Algebr. Geom. Topol.},
	volume={2},
	date={2002},
	pages={95--135},
	issn={1472-2747},
	review={\MR{1885217}},
	doi={10.2140/agt.2002.2.95},
}
	\bib{CJKLS}{article}{
	author={Carter, J. Scott},
	author={Jelsovsky, Daniel},
	author={Kamada, Seiichi},
	author={Langford, Laurel},
	author={Saito, Masahico},
	title={Quandle cohomology and state-sum invariants of knotted curves and
		surfaces},
	journal={Trans. Amer. Math. Soc.},
	volume={355},
	date={2003},
	number={10},
	pages={3947--3989},
	issn={0002-9947},
	review={\MR{1990571}},
	doi={10.1090/S0002-9947-03-03046-0},
}
\bib{CEGM}{article}{
	author={Churchill, Rasika},
	author={Elhamdadi, Mohamed},
	author={Green, Matthew},
	author={Makhlouf, Abdenacer},
	title={Ternary and $n$-ary $f$-distributive structures},
	journal={Open Math.},
	volume={16},
	date={2018},
	pages={32--45},
	issn={2391-5455},
	review={\MR{3760973}},
	doi={10.1515/math-2018-0006},
}
	\bib{ESZ}{article}{
	author={Elhamdadi, Mohamed},
	author={Saito, Masahico},
	author={Zappala,Emanuele},
	title={Higher Arity Self-Distributive Operations in Cascades and their Cohomology},
	status={preprint},	
	journal= {arXiv:1905.00440}
}
\bib{Vag}{book}{
	%			place={Cambridge}
	series={}
	volume = {}
	title={Wagner's Theory of Generalised Heaps}, 
	DOI={10.1007/978-3-319-63621-4}, 
	publisher={Springer International Publishing
	}, 
	author={Hollings, Christopher},
		author={Lawson, Mark}, year={2017}, 
}
\bib{Gr}{article}{
	author = {Grunspan, Cyril},
	year = {2002},
	month = {05},
	pages = {},
	title = {Quantum torsors},
	volume = {184},
	journal = {Journal of Pure and Applied Algebra},
	doi = {10.1016/S0022-4049(03)00066-5}
}
\bib{Kon}{article}{
	author={Kontsevich, Maxim},
	title={Operads and Motives in Deformation Quantization},
	journal={Letters in Mathematical Physics},
	year={1999},
	month={Apr},
	day={01},
	volume={48},
	number={1},
	pages={35--72},
	issn={1573-0530},
	doi={10.1023/A:1007555725247},
	url={https://doi.org/10.1023/A:1007555725247}
}


	\bib{NOO}{article}{
		author={Nelson, Sam},
		author={Oshiro, Kanako},
		author={Oyamaguchi, Natsumi},
		title={Local biquandles and Niebrzydowski's tribracket theory  },
		status={preprint},	
		journal= { arXiv:1809.09442 }
	}



	\bib{Maciej}{article}{
	author={Niebrzydowski, Maciej},
	author={Pilitowska, Agata},
	author={Zamojska-Dzienio, Anna},
	title={Knot-theoretic flocks},
	status={preprint},	
	journal= {arXiv:math/1908.09954v1}
}




	\bib{Sch}{article}{
	author={Schauenburg, Peter},
	title={Quantum torsors and Hopf-Galois objects},
	status={preprint},	
	journal= {arXiv:math/0208047}
}
\bib{Sk}{article}{
	author         = {Skoda, Zoran},
	title          = {Quantum heaps, cops and heapy categories},
	journal        = {Math. Commun.},
	volume         = {12},
	year           = {2007},
	pages          = {1-9},
	eprint         = {math/0701749}
%	archivePrefix  = {arXiv},
%	primaryClass   = {math},
%	SLACcitation   = "%%CITATION = MATH/0701749;%%"
}

\end{thebibliography}
\end{document}